\RequirePackage{amsthm}
\documentclass[pdflatex]{sn-jnl}% Math and Physical Sciences Numbered Reference Style 
\usepackage{graphicx}%
\usepackage{multirow}%
\usepackage{amsmath,amssymb,amsfonts}%
\usepackage{amsthm}%
\usepackage{mathrsfs}%
\usepackage[title]{appendix}%
\usepackage{xcolor}%
\usepackage{textcomp}%
\usepackage{manyfoot}%
\usepackage{booktabs}%
\usepackage{algorithm}%
\usepackage{algorithmicx}%
\usepackage{algpseudocode}%
\usepackage{listings}%
\usepackage{tikz, pgfplots}
\usetikzlibrary {arrows.meta}
\pgfplotsset{compat=1.18}
\theoremstyle{thmstyleone}%
\newtheorem{theorem}{Theorem}[section]
\newtheorem{proposition}[theorem]{Proposition}

\newtheorem{lemma}[theorem]{Lemma}

\newtheorem{q}[theorem]{Question}

\theoremstyle{thmstyletwo}%
\newtheorem{remark}[theorem]{Remark}%

\theoremstyle{thmstylethree}%

\AtEndDocument{%
  \par
  \medskip
  \textsc{Department of Mathematics, The Ohio State University, Columbus, OH, USA, 43210} \\
  \begin{tabular}{ll}%
    \hspace{12pt}
    (Gogolev) &
    \textit{Email address}: \texttt{gogolyev.1@osu.edu}\\
    \hspace{12pt}
    (Keck) &
    \textit{Email address}: 
    \texttt{keck.140@osu.edu}\\
    \hspace{12pt}
    (Lewis) &
    \textit{Email address}: \texttt{lewis.3164@osu.edu}
  \end{tabular}}

\begin{document}

\title[Article Title]{Bouncing Outer Billiards}

\author{\fnm{Andrey} \sur{Gogolev}}

\author{\fnm{Levi} \sur{Keck}}

\author{\fnm{Kevin} \sur{Lewis}}

\abstract{We introduce a new class of billiard-like system, ``bouncing outer billiards" which are 3-dimensional cousins of outer billiards of Neumann and Moser. We prove that bouncing outer billiard on a smooth convex body has at least four 1-parameter families of fixed points. We also fully describe dynamics of bouncing outer billiard on a line segment. Finally we carry out numerical experiments suggesting very complicated (non-ergodic) behavior for several shapes including the square and an ellipse. }

\maketitle

\section{Introduction}\label{sec1}

Outer billiards are dynamical systems introduced by Neumann in 1959 \cite{bib3} and then popularized by Moser in his lecture on stability of the solar system \cite{bib1, bib2}. The field of outer billiards became very active about 20 years ago. In this paper we suggest similar more complicated billiard systems called {\it bouncing outer billiards,} which we proceed to define. 

Let $S \subset \mathbb{R}^2$ be a compact convex set with smooth boundary. The {\it visibility domain}  $V_S$ consists of all  pairs $(p, v)$ where $p \in \mathbb{R}^2 \setminus int(S)$, and $v\in T^1_p\mathbb R^2$ is a unit vector based at $p$ such that the ray $\overrightarrow{R}$ spanned by $v$ has a non-empty intersection with $S$.

We now define the dynamical system $F_S\colon V_S\to V_S$ in the following way. Given an initial condition $(p,v)\in V_S$ the corresponding ray $\overrightarrow{R}$ reflects off the convex body at a point $w$ as $\overrightarrow{R}'$ in the usual way --- the angle of incidence equals to the angle of reflection.
Next we apply the outer  billiard law and consider the point $p'\in\overrightarrow{R'}$ such that $\|p-w\|=\|p'-w\|$ as indicated on Figure~\ref{fig:BOB_Dynamics}.

\begin{figure}
    \centering
    \includegraphics[width=0.8\linewidth]{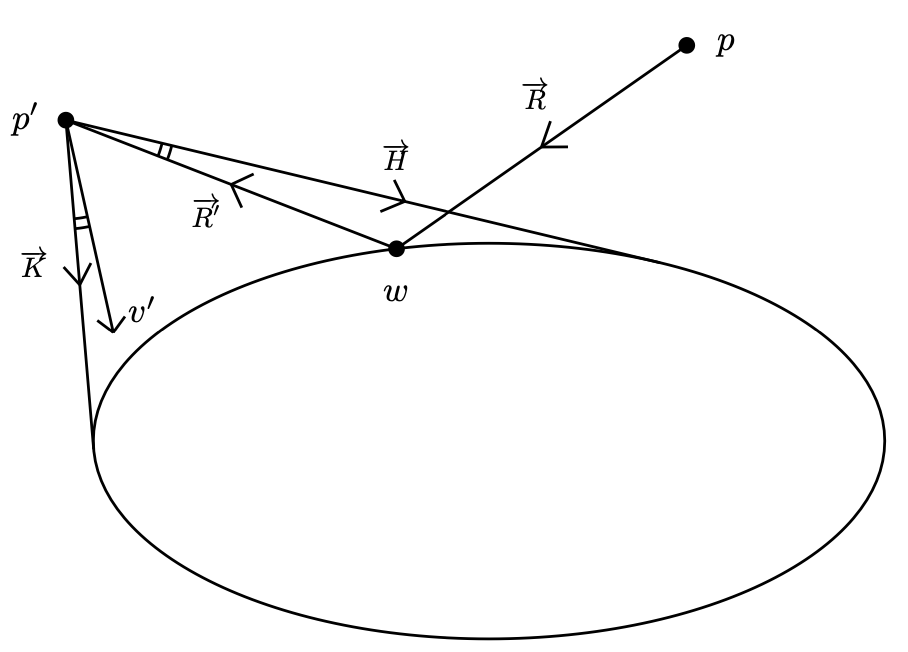}
  
    \caption{Bouncing Outer Billiards Dynamics}
    \label{fig:BOB_Dynamics}
\end{figure}

Finally, we will use the {\it visibility angle reflection rule} as follows. Let $\overrightarrow{H}$ and $\overrightarrow{K}$ be the rays at $p'$ which are tangent to $S$. Let $u$ be the unit vector based at $p'$ pointing to $w$ (in the direction opposite to $\overrightarrow{R}'$). Clearly $u$ is inside the angle defined by $\overrightarrow{H}$ and $\overrightarrow{K}$. Let $v'$ be the reflection of $u$ across the angle bisector of $\angle(\overrightarrow{H},\overrightarrow{K})$ as shown on Figure~\ref{fig:BOB_Dynamics}. This completes the definition of bouncing billiard dynamics.
$$
F_S(p,v)=(p',v').
$$
We will drop the subscript $S$ and simply write $F$ when no confusion is possible.

\begin{remark} 

It is easy to see that if $\overrightarrow{R}$ is tangent  to $S$ then we have the classical outer billiard dynamics. Hence the outer billiard is simply the restriction $F|_{\partial V_S}$ of the bouncing outer billiard to the boundary of the visibility space.

\medskip
    {\it\!\!\!\!\!\!\!\!\!\!\! Remark 1.2\,\,\,} If $S$ is not smooth, e.g. a polygon, then the angle reflection law is undefined for some initial conditions. However such initial conditions form a set of zero Lebesgue measure since the boundary of a convex body is differentiable almost everywhere. Hence bouncing billiard dynamics still makes sense for almost every initial condition, but the above relation to outer billiard is obscured.

    \medskip
{\it\!\!\!\!\!\!\!\!\!\!\!  Remark 1.3\,\,\,} S. Tabachnikov considered unfolding the outer billiard map into a family of symplectomorphisms given by the first two steps in the definition of the bouncing outer billiard~\cite{Tab}. However, to the best of our knowledge, the visibility angle reflection rule was not considered before.
\end{remark}

In the next section we establish existence of families of fixed points for bouncing outer billiards. Then we fully describe integrable twist map dynamics of bouncing outer billiards on a line segment. Finally, we present results of our numerical explorations in the last section.

We would like to pose two questions.

\begin{q}
    Does every orbit of the bouncing outer billiard on a smooth convex body remain bounded?
\end{q}
We were not able to detect any unbounded orbits numerically.

It is easy to check that bouncing outer billiards are conservative, that is, they preserve the Lebesgue measure on $V_S$ (see Appendix A).

\begin{q}
    Does there exist positive volume ergodic components?
\end{q}
We have found some orbits which appear to fill up 2-dimensional sets. However in the 3-dimensional spave $V_S$ such orbits seem to be confined to 2-dimensional surfaces.

\medskip

{\bfseries Acknowledgements.} This paper is a result of an REU project of summer 2024 at The Ohio State University. The authors are very grateful to Sergei Tabachnikov who provided several illuminating remarks on earlier drafts. The authors would like to acknowledge the support provided the NSF grant DMS-2247747. 

%--------------
%--------------

\section{Fixed points}
A natural question for any dynamical system is whether or not there exist fixed points, and if so, how to find them.

\begin{theorem}\label{fixedpointtheorem}
For any convex  $S$ with $C^3$ boundary the associated billiard map has uncountably many fixed points, which come in at least four local 1-parameter families.
\end{theorem}

Clearly, a point $(p,v)$ is fixed if and only if $v$ is the bisector of the angle formed by the tangent rays from $p$ to $S$. Therefore, given a point $p\notin S$, consider the angle given by the two tangent lines from $p$ to $S$ and let $v_p\in T_p\mathbb R^2$ be the vector spanning the angle bisector. The idea of the proof is to find a curve connecting two points, say, $p$ and $q$ such that the ray of $v_p$ ``bounces to the left'' and the ray of $v_q$ ``bounces to the right.'' Then, by the intermediate value theorem, there exists a fixed point $(r,v_r)$ on such curve.

%For any point $p$ outside of $S$, ``tangencies to $S$" will refer to the two unique lines containing $p$ and intersect $S$ but not its interior.  We will also shorten ``angle bisector of the tangencies to $S$ from the point $p$" to ``angle bisector at $p$" for brevity.

We note right away that if $\partial S$ has a circle subarc with constant curvature then then there is whole 2-parameter family of fixed points in proximity of such arc. Hence we can assume, due to $C^3$ regularity of the boundary that there exist arcs with strictly increasing and with strictly decreasing curvature. 

The following is our main lemma.

\begin{lemma}\label{fixedpointlemma}
Let $f\colon[s_0,s_2]\to\partial S$ be a counter-clockwise arc-length parametrzation of a subarc of $\partial S$ along which the curvature is strictly increasing. Assume that this arc is sufficiently short so that the tangent lines at $f(s_0)$ and $f(s_2)$ intersect at a point $p$ as indicated on Figure~\ref{fig:fixedpointlemmasetup}. 

Then the angle bisector ray spanned by $v_p$ will ``bounce off in the direction of $f(s_0)$,'' that is, after reflecting off $S$ the ray will intersect the tangent segment $a$.
\end{lemma}

%\begin{lemma}\label{fixedpointlemma}
%Let $f$ be an arc-length parameterization of $S$ with increasing positive curvature from parameter $s_0$ to $s_2$ (where the change in angle of $f'$ from $s_0$ to $s_2$ is less than $\pi$). Let $p$ be the unique point whose tangencies to $S$ intersect $f(s_0)$ and $f(s_2)$, respectively. A billiard coming from $p$ along the angle bisector will always bounce in the direction of lower curvature.

%Here, ``bounce in the direction of lower curvature" means that the final point after applying the billiards transformation will be on the same side of the angle bisector line as $f(s_0)$.
%\end{lemma}

\begin{proof}
Let $k(s)$ be the curvature at $f(s)$, and let $K(s)=\int_{s_0}^{s}k(t)dt$.  Since we are using arc-length parameterization, $K(s)$ is the angle between the tangent lines at $f(s_0)$ and $f(s)$.
There is a unique $s_1$ such that $K(s_1)={K(s_2)}/{2}$. Then the tangent line at $f(s_1)$ is perpendicular to $v_p$. Hence to prove the claim of the lemma it suffices to show that the distance from $f(s_1)$ to the tangent line $b$ is less than the distance from $f(s_1)$ to the tangent line $a$. This can be expressed by the following inequality:

%If $p$ is a fixed point, then bouncing along the angle bisector must hit $S$ at $f(s_1)$, as this is the only point where the tangent angle is perpendicular to the angle bisector from $p$.

\begin{figure}
    \centering
    \includegraphics[width=1\linewidth]{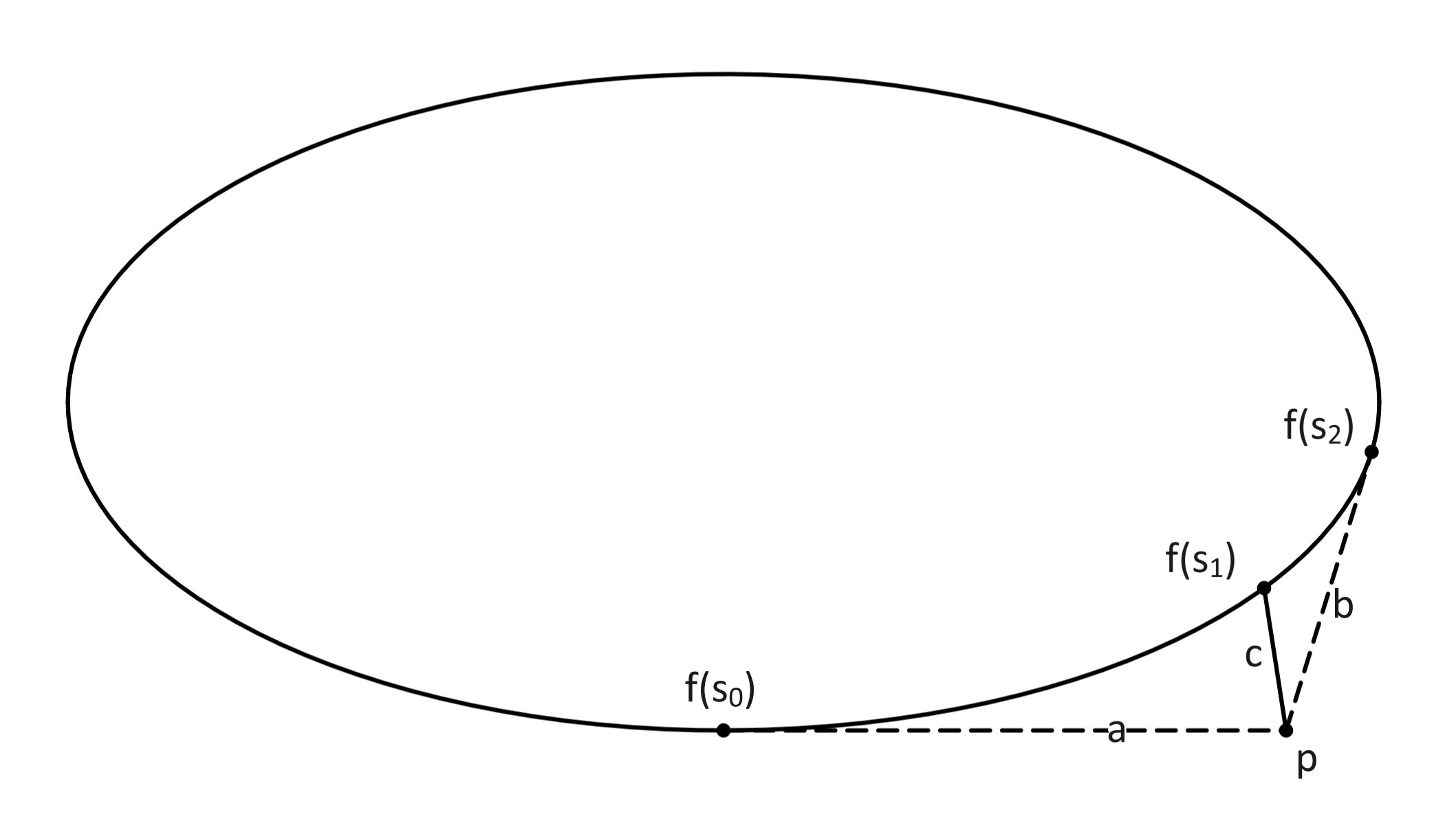}
    \caption{Fixed Point Lemma Setup}
    \label{fig:fixedpointlemmasetup}
\end{figure}

%As shown in the example setup in Figure~\ref{fig:fixedpointlemmasetup}, we will show that the angle between $c$ and $b$ is less than the angle between $c$ and $a$, so when bouncing along the actual angle bisector, we will bounce off a point between $f(s_0)$ and $f(s_1)$. The tangent angle here will be too close to the tangent angle at $f(s_0)$ to be perpendicular to the angle bisector from $p$, so we will bounce ``towards" $f(s_0)$

%To show this angle difference, we will show that the distance from $f(s_1)$ to the line $b$ is less than the distance from $f(s_1)$ to $a$. This can be expressed by the following inequality:

\begin{equation*}
    \int_{s_1}^{s_2}\sin(K(s_2) - K(s))ds < \int_{s_0}^{s_1}\sin(K(s))ds \tag{$\ast$}
\end{equation*}
To prove this inequality, we can start with the following statements by change of variables:
\begin{equation*}
\begin{aligned}
\int_{s_1}^{s_2}\sin(K(s_2) - K(s))k(s)ds &= \int_{0}^{K(s_1)}sin(u)du \\
\int_{s_0}^{s_1}\sin(K(s))k(s)ds &= \int_{0}^{K(s_1)}sin(v)dv
\end{aligned}
\end{equation*}
This gives
\begin{equation*}
    \int_{s_1}^{s_2}k(s)\sin(K(s_2) - K(s))ds = \int_{s_0}^{s_1}k(s)\sin(K(s))ds
\end{equation*}
Since curvature $k\colon[s_0,s_2]\to\mathbb R_+$ is increasing the posited inequality follows proving the lemma.
\end{proof}
%Since curvature is increasing from $s_0$ to $s_2$, $k(s)$ is always greater than $k(s_1)$ on the LHS, and $k(s)$ is always lesser than $k(s_1)$ on the RHS.

%So (since both the integrand and $k(s)$ are always positive),
%\begin{equation*}
%    k(s_1)\int_{s_1}^{s_2}\sin(K(s_2) - K(s))ds < k(s_1)\int_{s_0}^{s_1}\sin(K(s))ds
%\end{equation*}

%Since $k(s_1)$ is positive, this implies %($\ast$), proving the lemma.

\begin{proof}[Proof of Theorem~\ref{fixedpointtheorem}.]

Consider a local minimum (or maximum) of the curvature of $\partial S$. (By the 4-vertex theorem at least four local extrema exist.) On one side there is a short arc with increasing curvature and on the other side there is short arc with decreasing  curvature. Applying the above lemma to the first arc we obtain an initial condition $(p,v_p)$ which ``bounces to the left'' and similarly, applying (the analogue of) the lemma to the arc with increasing curvature we obtain an initial condition $(q, v_q)$ which ``bounces to the right.'' It remains to connect $p$ and $q$ by an arc disjoint with $S$ an apply the intermediate value theorem.
\end{proof}

%Consider the continuous mapping $f$ from points outside of $S$ to the signed angle difference between the outward normal and the incident ray when bouncing along the angle bisector of tangencies to $S$. By Lemma~\ref{fixedpointlemma} and a symmetrical argument for decreasing curvature, we can find a point where $f$ is negative and a point where $f$ is positive (unless $S$ is a circle, where $f$ is always 0 and all points are fixed points when bouncing along the angle bisector). This is due to the fact that the curvature of $f$ is differentiable, and we can use the four-vertex theorem.

%So, by the intermediate value theorem, any curve between two points that bounce in opposite directions must contain a point where $f$ is 0, which defines a fixed point. We can draw uncountably many curves between these two points that only intersect at these points themselves. and since these endpoints are clearly not fixed points, we must get distinct fixed points for each curve, defining uncountably many fixed points of our system.
%\end{proof}

%Therom~\ref{fixedpointtheorem}, when combined with continuity, almost implies that there is a one parameter family of fixed points extending outwards from each vertex (local extremum of curvature) along some curve, and that these curves extend outwards infinitely far without meeting. However, it has not been proven that these curves actually do not meet, or that such a family coming from one local vertex will not terminate at another.

It is clear from the above proof that each vertex of $\partial S$ yields a 1-parameter family of fixed points. These families could merge away from $S$. We would like to pose the following question.
\begin{q}
Let $S$ be a convex domain with $C^3$ boundary. Does every closed curve around it contain at least one fixed point of the bouncing outer billiard map on $S$?
\end{q}

S. Tabachnikov considered a question of very similar flavor and eventually found a counterexample~\cite{bib4}.

\section{Bouncing on a Line Segment}\label{secLS}

\subsection{Parameterizing the Dynamics}\label{subsecPD}

This section focuses on the behavior of the bouncing outer billiards system on a line segment. Since all segments are congruent up to scaling, we only consider the segment on the $x$-axis from -1 to 1. For convenience, we will consider only initial points $p$ with positive $y$-values, as points with negative $y$-values are symmetric.

Recall that we denote the visibility domain by $V$. Consider initial condition $(p, v)~\in~V$, where $p = (x, y)$, and let $\theta = \arg(v) + \frac{\pi}{2}$. The initial conditions define a ray from the point $p$ with slope $\tan(\theta - \frac{\pi}{2})$. Using this equation, we can derive:
\begin{equation*}
    w = (x + y  \tan(\theta), 0)
\end{equation*}
\begin{equation*}   
    p' = (x', y') = (x + 2y\tan(\theta), y)
\end{equation*}

Note that the $y$-coordinate of the initial point will remain constant on the orbit. From this point forward, we will refer to the $y$-value of the initial point as height and denote it by $h$.

Next, we obtain that the angle from $p'$ to the left and right endpoints of the segment are given by $\arctan\left(\frac{1-x'}{h}\right)$ and $\arctan\left(\frac{-1-x'}{h}\right)$, respectively. Also, the angle from $p'$ to $w$ is given by $-\theta$. Applying the visibility angle reflection rule yields:

\begin{equation*}
\begin{aligned}
\theta' &= \arctan\left(\frac{1-x'}{h}\right) + \theta  +\arctan\left(\frac{-1-x'}{h}\right)
\end{aligned}
\end{equation*}
\vspace{0.3 cm}
Summing up, the dynamics, $F(x, h, \theta) = (x', h, \theta')$ is given by:

\begin{equation}\label{dynamics}
\begin{gathered} 
    x' = x + 2h  \tan(\theta) \\
    \theta' = \theta + \arctan\left(\frac{1-x'}{h}\right) + \arctan\left(\frac{-1-x'}{h}\right)
\end{gathered}
\end{equation}

\subsection{A Second Invariant}\label{subsec22}

In Section~\ref{subsecPD}, we observed that the height $h$ is an invariant. In this section, we will demonstrate the existence of a second invariant.

We define a new coordinate system, relative to which the it's easier to see a second invariant. First consider the change of coordinates $g(x, h, \theta) = (w, h, d)$ given by

\begin{equation}\label{transform1}
    \begin{cases}
        w = x + h \tan(\theta) \\
        d = h \tan(\theta)
    \end{cases}
\end{equation}
with the $h$-coordinate remaining unchanged. The coordinate $w$ represents the $x$-value of the bounce point and the coordinate $d$ represents the signed difference between the $w$ and the $x$-value of the initial point. The inverse coordinate transformation is given by:

\begin{equation}\label{inverse_transform_1}
    \begin{cases}
        x = w-d \\
        \theta = \arctan\left(\frac{d}{h}\right)
    \end{cases}
\end{equation}

Now, we seek to understand dynamics in these new coordinates. Let $f(w, h, d) = g \circ F \circ g^{-1}(w, h, d)$. We let $f(w, h, d)$ be denoted by $(w', h, d')$, which we wish to write in terms of $w$, $h$, and $d$. First, sing~(\ref{dynamics}) and~(\ref{transform1}) yields:
\begin{equation} \label{coordinate_relation}
    x' = w + d
\end{equation}

Now, we will use (\ref{dynamics}) to rewrite the equation for $d'$. Following this, we simplify and use (\ref{transform1}) and (\ref{coordinate_relation}) to rewrite all instances of $x'$ and $h  \tan(\theta)$ in terms of $w$ and $d$.
\begin{equation} \label{d'}
\begin{aligned}
    d' &= h  \tan(\theta') \\ 
    &= h  \tan\left(\theta + \arctan\left(\frac{1-x'}{h}\right) + \arctan\left(\frac{-1-x'}{h}\right)\right) \\
    &= h  \left(\frac{\tan(\theta) + \frac{-2x'h}{1+h^2-(x')^2}}{1+ \frac{2x'h}{1+h^2-(x')^2}}\right) \\
    &= \frac{h  \tan(\theta) + h^3  \tan(\theta) - h(x')^2  \tan(\theta) - 2x'h^2}{1+h^2-(x')^2 + 2x'h  \tan(\theta)} \\
    &= \frac{d^3 + 2h^2w+2d^2w+dw^2+h^2d-d}{w^2-d^2-h^2-1}
\end{aligned}
\end{equation}
\vspace{0.3 cm}
Finally, we can calculate $w'$ by using the relationship $w' = d' + x'$, which gives:
\begin{equation}\label{w'}
\begin{aligned}
    w' &= w + d + \frac{d^3 + 2h^2w+2d^2w+dw^2+h^2d-d}{w^2-d^2-h^2-1} \\
    &= \frac{w^3+d^2w+h^2w+2dw^2-w-2d}{w^2-d^2-h^2-1}
\end{aligned}
\end{equation}

\vspace{0.3 cm}

We will now show the existence of a second invariant denoted $a^2$, as $a$ will later be shown to be the semi-axis of an ellipse.

\begin{proposition}\label{propA}
    Quantity $a^2 := \frac{h^2w^2+d^2}{h^2+d^2}\in(-1,1)$ is preserved under dynamics. That is, 
    $$\frac{h^2w^2+d^2}{h^2+d^2} = \frac{h^2(w')^2+(d')^2}{h^2+(d')^2}.$$
\end{proposition}

The proof involves substituting $w'$ and $d'$ into the equation for $a^2$ to get $(a')^2 = \frac{h^2(w')^2 + (d')^2}{h^2 + (d')^2}$. After using (\ref{d'}) and (\ref{w'}) to simplify, we obtain:

\begin{equation*}
\begin{aligned}
    (a')^2 &= \frac{(h^2w^2+d^2)(p(w, h, d))}{(h^2+d^2)(p(w, h, d))} \\
    &= \frac{h^2w^2 + d^2}{h^2+d^2}=a^2,
\end{aligned} 
\end{equation*}
\medskip 
where 
\begin{multline*}
p(w, h, d) = d^4+h^4+2d^2h^2+4d^3w+6d^2w^2+4dh^2w\\
+4dw^3+w^4+2h^2w^2-4dw-2w^2+1+2h^2-2d^2.
\end{multline*}
\qed

Related to this is an equivalent invariant
\begin{equation*}
    b^2 = \frac{h^2w^2+d^2}{1-w^2} = \frac{h^2a^2}{1-a^2}.
\end{equation*}

\subsection{Invariant Ellipses}\label{subsec23}

In our altered coordinate system, the invariants $a$ and $b$ are actually the semi-axes of an invariant ellipse in the $(w, h, d)$ coordinate system.

\begin{proposition}\label{propE}
    Let $w, h, d \in \mathbb{R}$. Recalling the definitions $a^2 = \frac{h^2w^2+d^2}{h^2+d^2}$ and $b^2 = \frac{h^2w^2+d^2}{1-w^2}$, we have $\frac{w^2}{a^2} + \frac{d^2}{b^2} = 1$ (when $a^2, b^2 \ne 0$). 
\end{proposition}

\begin{proof}
\begin{equation*}
\begin{aligned}
    \cfrac{w^2}{a^2} + \cfrac{d^2}{b^2} &= \cfrac{w^2}{\cfrac{h^2w^2 + d^2}{h^2+d^2}} + \cfrac{d^2}{\cfrac{h^2w^2+d^2}{1-w^2}} \\
    &= \cfrac{h^2w^2 + d^2w^2 + d^2 - d^2w^2}{h^2w^2 + d^2} \\
    &= 1
\end{aligned}
\end{equation*}
\end{proof}

\begin{figure}[hbt!]
    \centering
    \includegraphics[width=0.6\linewidth]{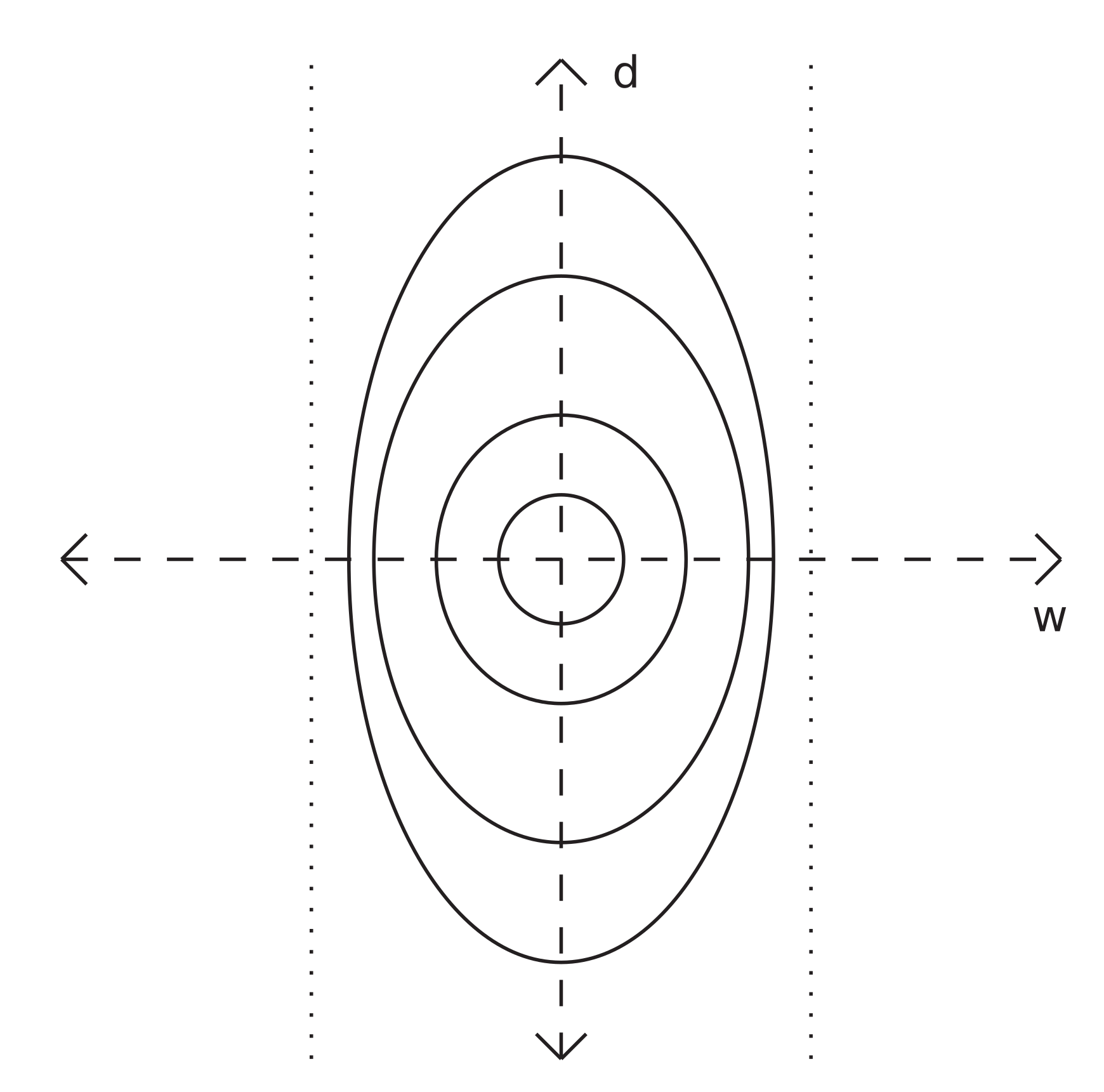}
    \caption{Several Invariant Ellipses with Height One}
    \label{fig:ellipses}
\end{figure}

By Proposition~\ref{propA}, we have that $a^2 = (a')^2$, and by the relationship between $a$ and $b$ we have that $b^2 = (b')^2$. Along with Proposition~\ref{propE}, we get

\begin{equation*}
    \frac{(w')^2}{a^2} + \frac{(d')^2}{b^2} = 1,
\end{equation*}
thus showing that any orbit belongs to an ellipse in the $(w, d)$-coordinate system.

Note that if $a$ or $b$ are equal to zero, then the other must be as well by the equation relating them. If they are both zero, we have that $w=d=0$ for all points in the orbit. Using (\ref{inverse_transform_1}), this implies that $x=\theta=0$ for all points in the orbit, which means such initial conditions correspond to fixed points.
 
\subsection{Twist Dynamics}\label{subsec24}

For this section, we will fix a height $h$ and an invariant ellipse, thereby fixing invariants $a$ and $b$, which are defined to be the positive square roots of $a^2$ and $b^2$ respectively. We can parameterize the ellipse with $r(\theta) = (w, d) = (a  \cos(\theta), b  \sin(\theta))$. We now define the function $\overline{f} : S^1 \rightarrow S^1$ as $\overline{f} = r^{-1} \circ f \circ r$, which allows us to view the restriction of $f$ to our invariant ellipse as a circle diffeomorphism.

\begin{theorem}\label{thmP}
    There exists some $\varphi \in S^1$ such that $\overline{f}(\theta) = \theta + \varphi$, where $\varphi = \varphi(a)$ is a strictly increasing function, $\varphi'(a) > 0$ given by:
    \begin{equation*}
    \varphi(a) = \begin{cases}
        \arctan(\frac{2ab}{b^2-a^2}) + \pi & a < b \\
        \arctan(\frac{2ab}{b^2-a^2}) & a > b \\
        \frac{3\pi}{2} & a=b        
    \end{cases}
\end{equation*}
with
\begin{equation*}
    \varphi'(a) = \cfrac{2b}{b^2+a^2}.
\end{equation*}
\end{theorem}

The proof of this theorem is computational and will be included in the Appendix~B. 

\subsection{Periodic Orbits for the Billiard on the Segment}\label{subsec25}

Clearly the middle perpendicular, which is the $y$-axis gives a 1-parameter family of fixed point (these correspond to degenerate ellipses with $a=b=0$). Points of higher period are due to rational rotation numbers and come in 2-parameter families as one can vary the height as well. 

\begin{figure}[H]
    \centering
    \includegraphics[width=0.45\linewidth]{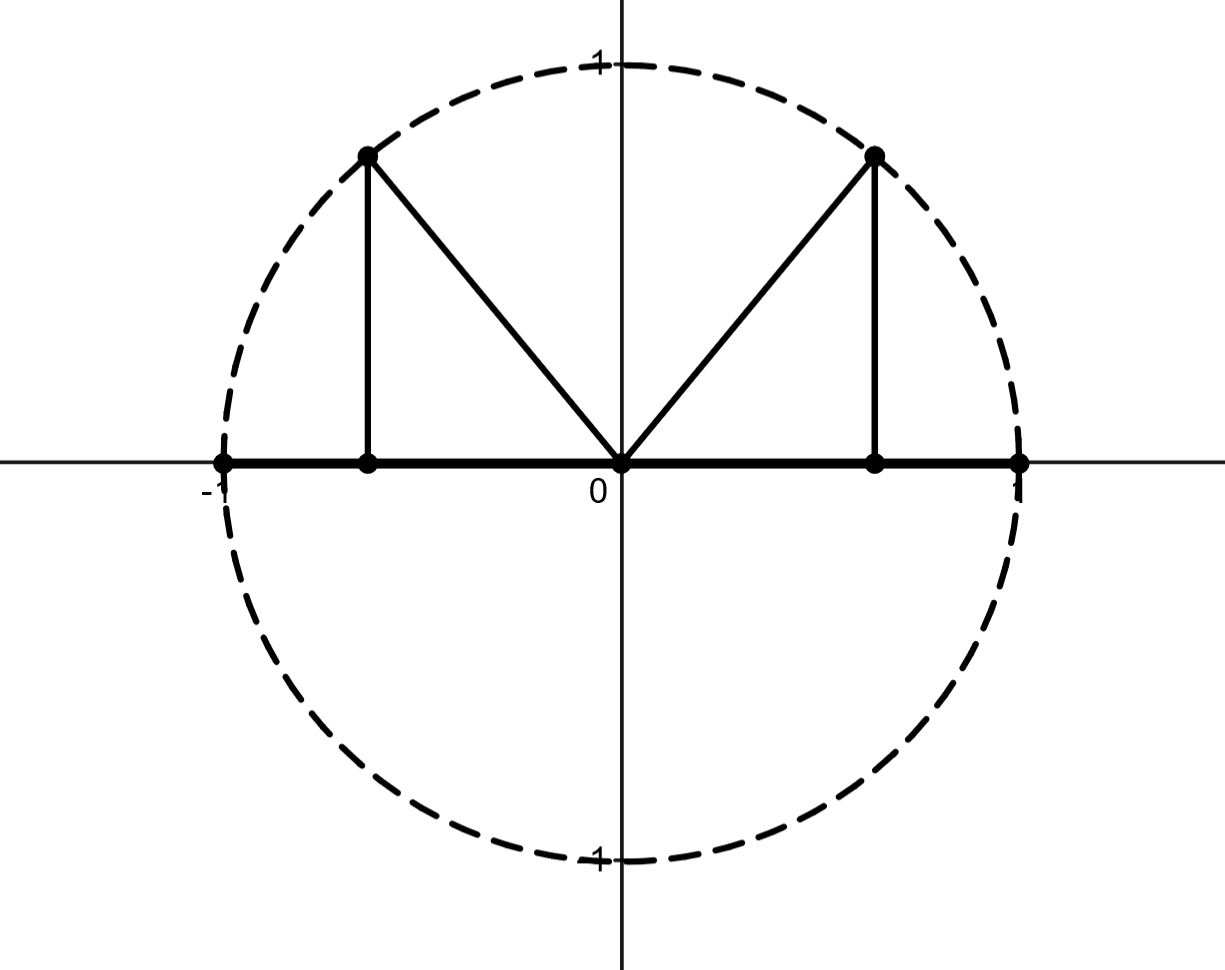}\hfill
    \includegraphics[width=0.45\linewidth]{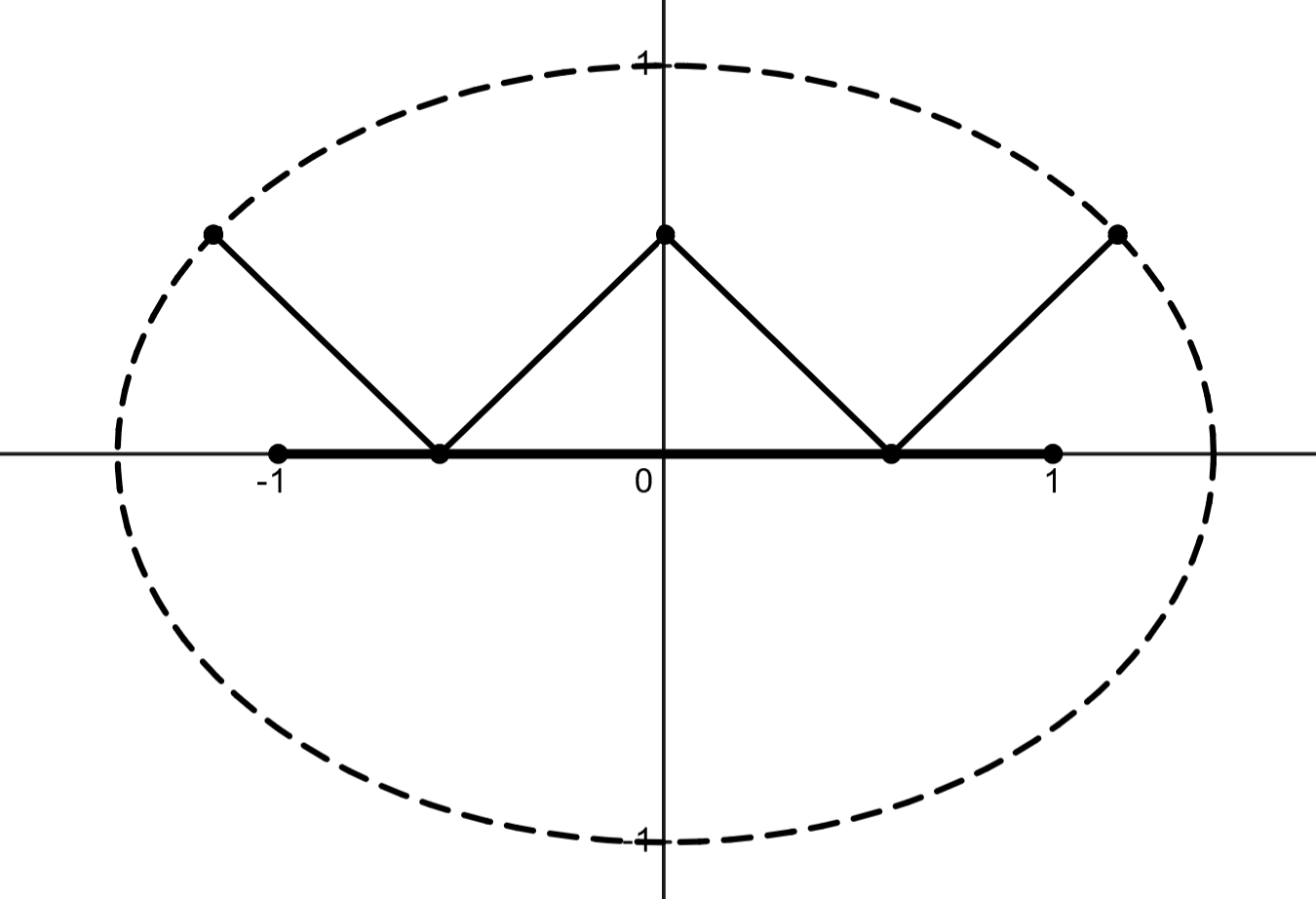}
    \caption{``M" and ``W" period 4 orbits}
    \label{fig:MsAndWs}
\end{figure}

We point out two aesthetically pleasing sub-families of period 4 orbits on Figure~\ref{fig:MsAndWs}. The $M$-orbits fill out a semi-circle and the $W$-orbits fill out a semi-ellipse. For the ``W" case on the right, the horizontal semi-axis is $\sqrt{2}$ and the foci of this ellipse are the $\pm 1$ endpoints of the segment.

%Now that we have developed the understanding of the system on the segment as a constant rotation of an ellipse parameter in $w$ and $d$ coordinates, we are in a position to make note of which periodic orbits exist on the segment.

%Before this, there is one special case of the starting conditions which we have yet to address. If both $w$ and $d$ are zero, we end up getting that $a$ and $b$ are both zero. In this case, our previous analysis doesn't hold, but this case behaves very simply. This is the case in which the initial $x$-value is zero, and the initial angle is zero, meaning the bounce occurs at $x=0$ and $(x
%, \theta) = (x', \theta')$. These starting conditions ($h$ can be any value) represent the only fixed points of the system. This is because if either $a$ and $b$ have a nonzero value, then the other must as well by the equation relating them, and $\arctan(\frac{2ab}{b^2-a^2})$ is only zero when either $a$ or $b$ is zero.

%Next, note that the rotation $\varphi$ is restricted to the range $(\pi, 2\pi)$ when $a \ne 0$ and $b \ne 0$. This is because $\arctan(\frac{2ab}{b^2-a^2})$ outputs in the range $(\frac{3\pi}{2}, 2\pi)$ when $a > b$ and $\arctan(\frac{2ab}{b^2-a^2}) + \pi$ outputs in the range $(\pi, \frac{3\pi}{2})$ when $a < b$. Finally, the rotation is $\frac{3\pi}{2}$ when $a = b$. This range means that any periodic orbit of least period 2 is impossible, since it would correspond to a rotation of $\pi$, which is not a possible rotation value. However, this does mean that periodic orbits of least period $n$ for every $n > 2$ exist. 

It is easy to calculate from the formula for the rotation number in Theorem~1.3 that for a given height $h$ the interval of possible rotation numbers $\varphi$ has the form $(\pi, \rho(h))$, where $\rho$ is an explicit decreasing function, $\rho(h)\to 0$, $h\to\infty$; $\rho(h)\to2\pi$, $h\to 0$. In particular, $(\pi, \rho(h))\subset (\pi, 2\pi)$ and, hence, there are no orbits of period 2. Clearly, for all sufficiently small heights orbits of all periods $\ge 3$ are present. As height increases smaller period orbits begin to disappear. For example orbits of period 4 with rotation number $\frac{3\pi}{2}$ disappear at $h=1$ and orbits of period 3 with rotation number $\frac{4\pi}{3}$ disappear at $h\simeq 1.8$.

\begin{figure}[hbt!]
    \centering
    \includegraphics[width=1\linewidth]{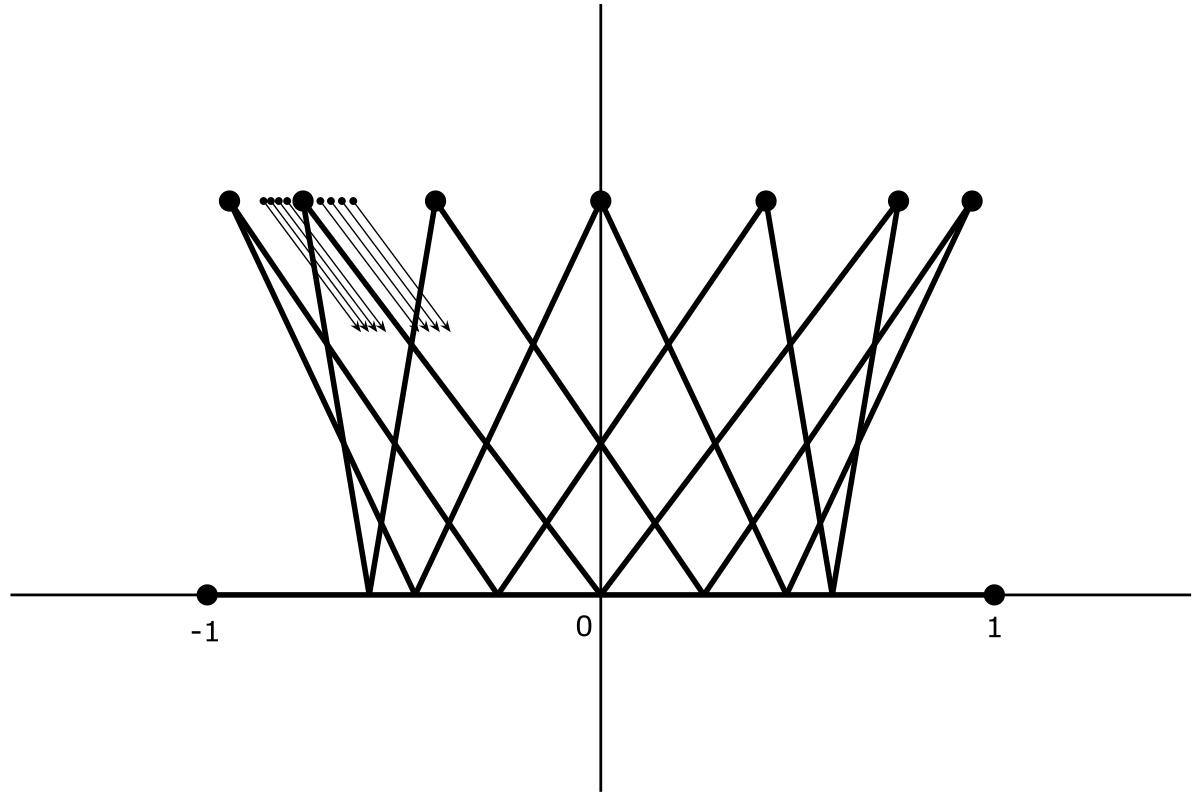}
    \caption{Family of Orbits of Period 7}
    \label{fig:least-period-seven}
\end{figure}

The explicit formula of Theorem~1.3 allows to explicitly calculate periodic orbits.
For example, if one wished to find  periodic orbits of least period 7, one can calculate parameter values which correspond to the rotation number $\frac{10\pi}{7}$. 

%Since $\frac{10\pi}{7} < \frac{3\pi}{2}$, we know we must have $a < b$ with $\varphi = \arctan(\frac{2ab}{b^2-a^2}) + \pi$, meaning $\arctan(\frac{2ab}{b^2-a^2}) = \frac{3\pi}{7}$. At this point, we have some freedom in choosing parameters. Let us take $a = \frac{3}{4}$. Solving for $b$ yields $b \approx 0.9405$. Accordingly, $h \approx 0.8294$. These parameters define the invariant ellipse with rotation number $\frac{10\pi}{7}$. 
%We need to choose some initial condition on this ellipse. For simplicity, we'll choose $w = 0$, $d = b$. Finally, we can convert $w$ and $d$ to $x$ and $\theta$ values using the inverse coordinate transformation given earlier. This yields $x = -b$, $\theta = \arctan(\frac{b}{h})$. 
Figure~5 depicts a family of period seven orbits of height one.
Note the depicted orbit is symmetric about the line $x = 0$, which unfolds into the family of asymmetric period seven orbits as indicated on the figure.

\vspace{10pt}

\section{Numerical Simulations}
\subsection{Bouncing on Parabola Arcs}\label{subsecParabola}

After fully understanding the dynamics on the segment, we can perturb the dynamics and examine how integrability is being destroyed. Probably the simplest ways is to consider the unfolding of the segment into a piece of a downward-facing parabola given by $f(x) = -ax^2 + a \hspace{0.2 cm} \{-1 \le x \le 1\}$. We will make a slight modification to our visibility domain to make sure that bouncing billiard still makes sense.

Recall that according to our definition we required that for $(p,v)$ in the visibility domain, the ray spanned by $v$ has a nonempty intersection with the boundary of the set. For the parabola, we will require that the ray has a nonempty intersection with the parabola, but will impose the additional requirement that the segment $\overline{pw}$ lies entirely above the parabola given by $-ax^2 + a$, where $w$ is the closest intersection point to $p$ of the ray and parabola. In other words, a point $p$ is not be able to ``see" the underside of the parabola. The dynamics rule remains the same, simply utilizing the newly defined visibility domain for the visibility angle reflection.

\begin{remark}
Despite the fact that our integrable model is a perfect twist map, KAM theory doesn't apply directly since we are in a 3-dimensional situation. Still as we see below KAM features such as elliptic islands seem to be present in our unfolding.
\end{remark}

Figure~\ref{fig:parabola-height-three-tenths} depicts some orbits on a parabola of height $\frac{3}{10}$. We observe that most of the orbits which begin close to the parabola fill up invariant circles which align themselves along the parabola such as the green orbit in the figure. Others, such as the red and blue fill up periodic curves. Finally, some orbits, such as the purple one exhibit more complicated behavior similar to  Aubry-Mather sets with positive Lyapunov exponent.

\begin{figure}[hbt!]
    \centering
    \includegraphics[width=0.9\linewidth]{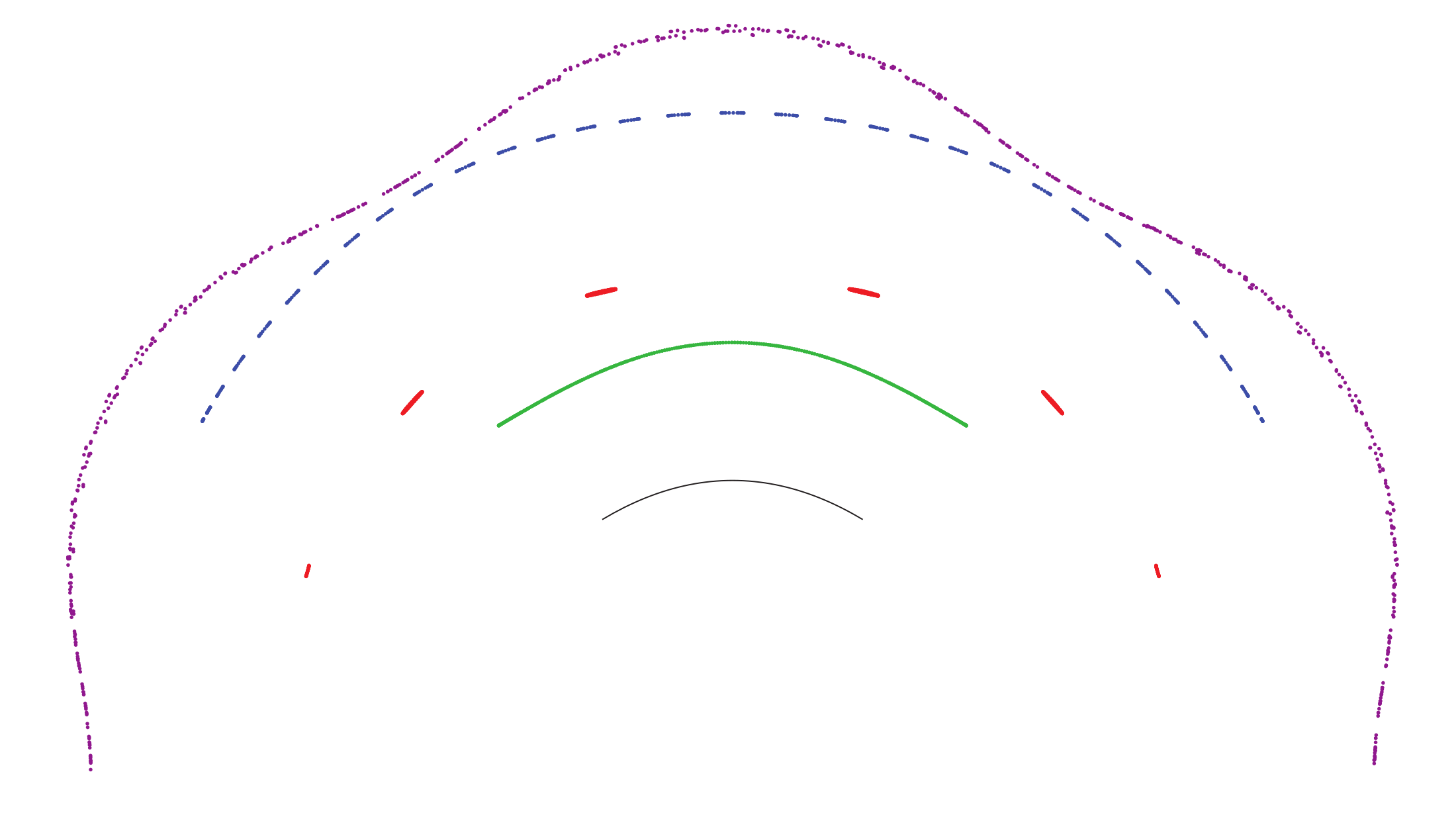}
    \caption{Orbits on Parabola of Height $\frac{3}{10}$}
    \label{fig:parabola-height-three-tenths}
\end{figure}

\begin{figure}[hbt!]
    \centering
    \includegraphics[width=0.9\linewidth]{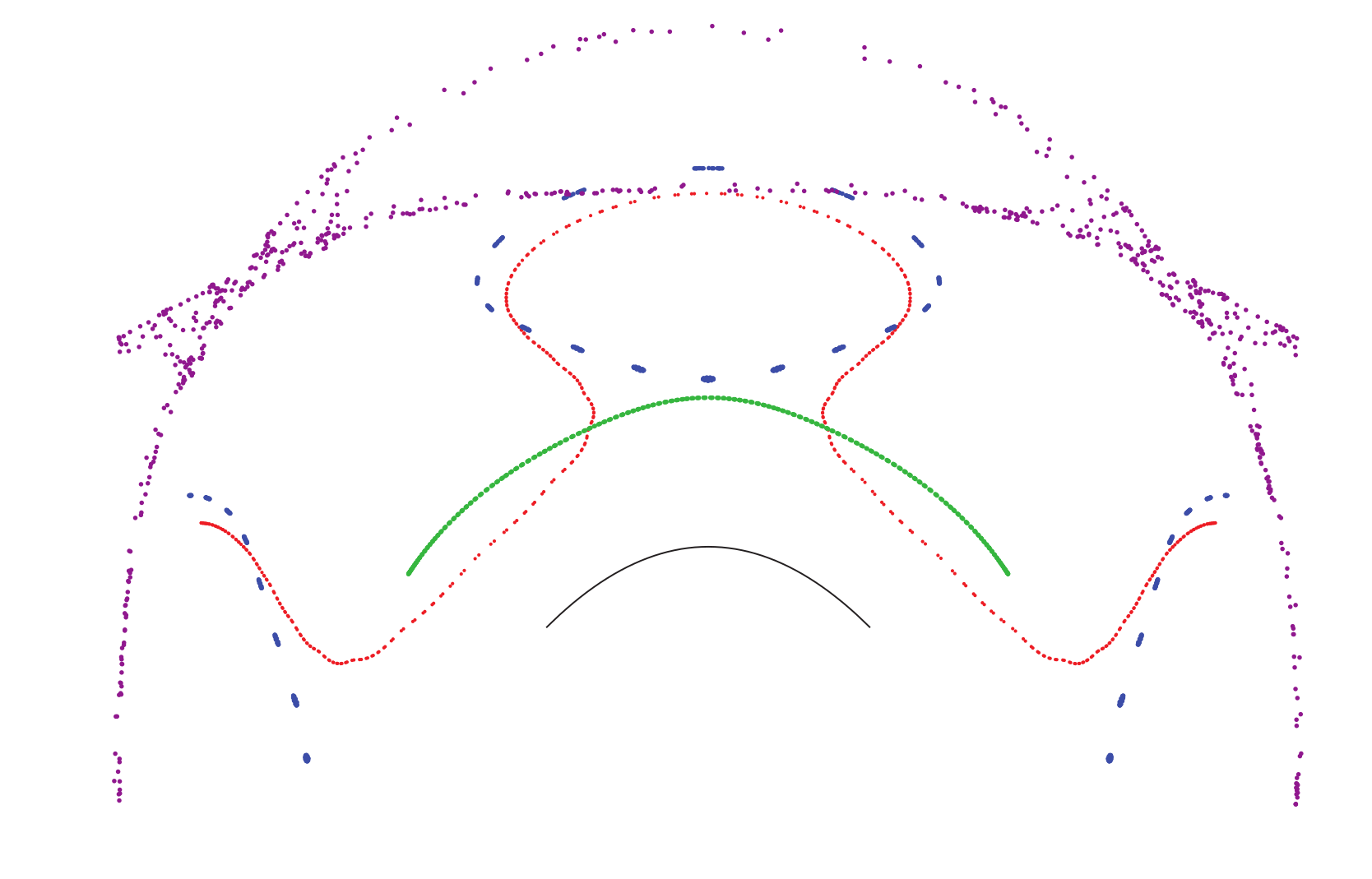}
    \caption{Orbits on Parabola of Height $\frac12$}
    \label{fig:parabola-height-one-half}
\end{figure}

\begin{figure}[hbt!]
    \centering
    \includegraphics[width=0.9\linewidth]{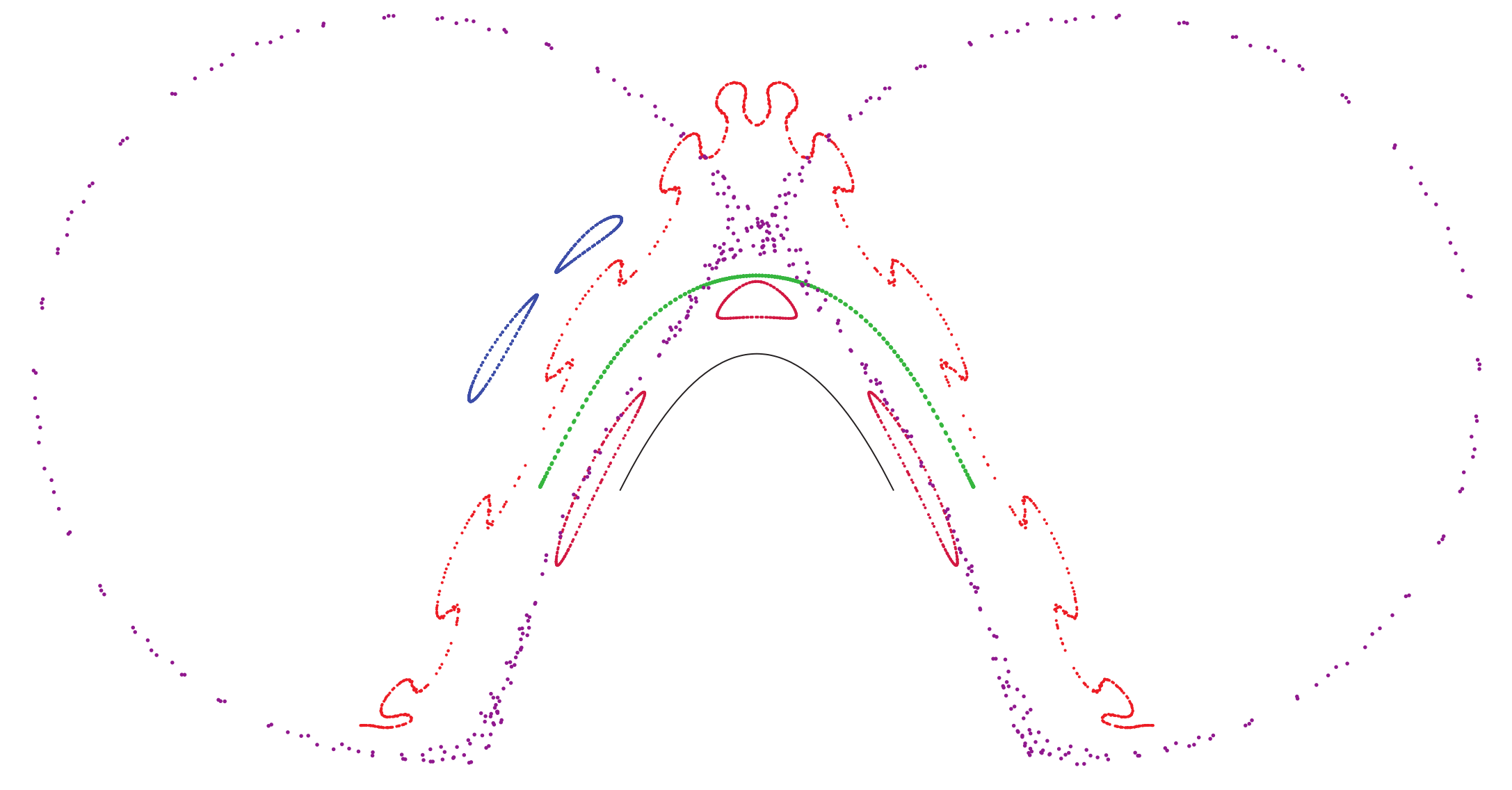}
    \caption{Orbits on Parabola of Height 1}
    \label{fig:parabola-height-one}
\end{figure}

As we increase the height, the observed behaviors become more complicated. Figure~\ref{fig:parabola-height-one-half} and Figure~\ref{fig:parabola-height-one} depict orbits on parabolas of heights $\frac{1}{2}$ and 1, respectively. On these more extreme parabolas, we still observe periodic curves which take more complicated shapes, including non-symmetric such the blue orbit on Figure~\ref{fig:parabola-height-one} . 
%We observe more interesting configurations of orbits lying on many closed curves, such as the red and blue orbits of Figure~\ref{fig:parabola-height-one-half}. 
Additionally, with increased height we more easily detect chaotic behavior such as the dark purple orbits of both figures and the red orbit of Figure~\ref{fig:parabola-height-one}.

\subsection{Bouncing on the Square}\label{subsecSquare}

It is also interesting to investigate bouncing billiards on polygons. For the sake of simplicity, we will focus solely on the system on the square. On the square, we classify observed orbits into five categories.

\begin{figure}[hbt!]
    \centering
    \includegraphics[width=0.67\linewidth]{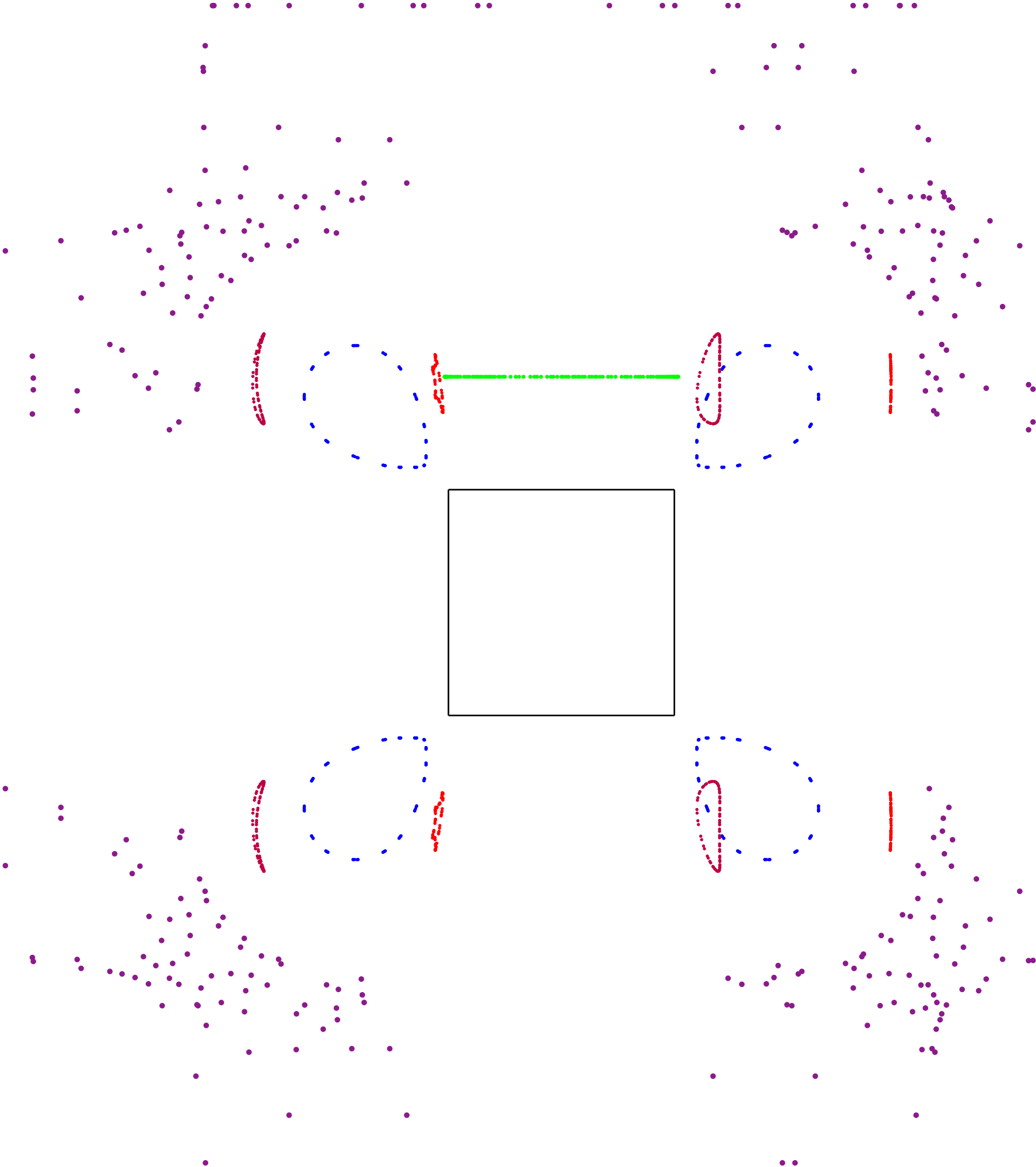}
    \caption{Orbits on a Square}
    \label{fig:square}
\end{figure}

The first, such as the green orbit on Figure~\ref{fig:square} consists points staying a fixed perpendicular distance away from one side of the square. For some such orbits, the orbit never extend in the direction parallel to that side further than the endpoints of the side. In this case, the system is identical to that on the line segment. In other cases, such orbit can extend past the corners of the square while still remaining on one side; in this case the orbit is not the same as an orbit on the segment. 

The second kind are those orbits which fill up four closed curves, with one near each corner of the square. This can be seen in the magenta orbit of Figure~\ref{fig:square}. Most of these orbits observed appeared to be rotationally symmetrical, but the one pictured is not.  

The third kind involves what appears to be a invariant circle near each corner, but actually consists of many smaller closed curves making up the apparent larger circle. This is depicted in the blue orbit of the figure. 

The fourth kind is another chaotic variety. It involves a period 4 non-smooth set, possibly a Cantor set. This kind is depicted in the red orbit of the figure.

The final class of orbits occupy all sides of the square and seem to behave chaotically such as the dark purple orbit. Such orbits appear to fill in a positive area domains. However, numerics becomes very tricky for such orbits as we clearly detected positive Lyapunov exponent for such orbits. 

\begin{figure}[H]
    \centering
    \includegraphics[width=0.5\linewidth]{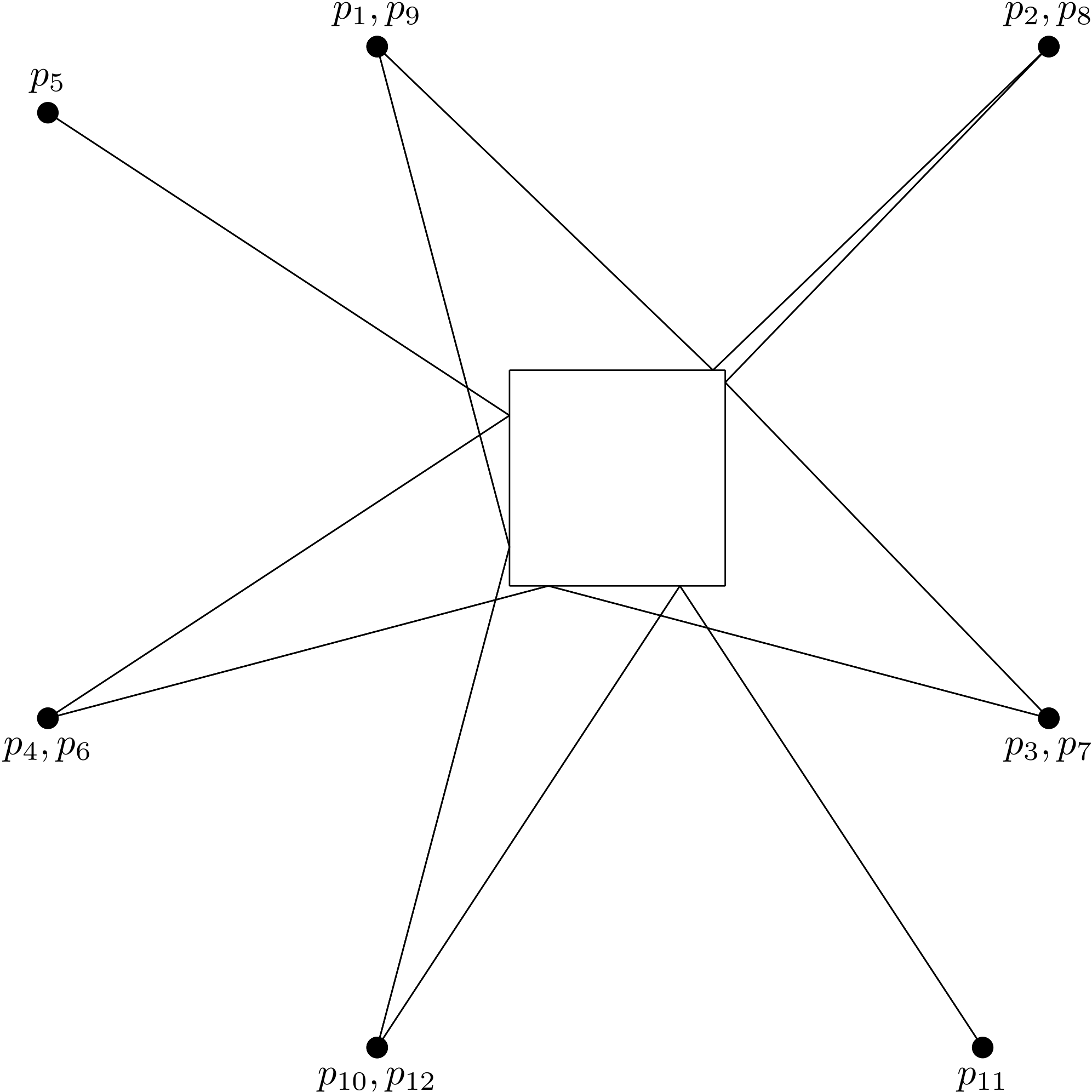}
    \caption{Period Twelve Orbit on Square with Large Eigenvalue}
    \label{fig:eigenvalues}
\end{figure}

Another finding on the square is the existence of periodic points whose Jacobian matrix has eigenvalues greater than one. Figure~\ref{fig:eigenvalues} shows one such example. It depicts a period twelve orbit whose Jacobian matrix has eigenvalues approximately 0.086, 1, and 11.592.

\begin{remark}
    It is easy to see from the form of the differential of the bouncing outer billiard on a convex polygon that every periodic point of such billiard has at least one eigenvalue equal to 1.
\end{remark}

\subsection{Bouncing on an Ellipse}\label{subsecEllipse}

While we fully understand the segment and the circle, in between fall ellipses, which also show very complex behavior. We consider bouncing outer billiard on the ellipse with major and minor semi-axis equal to 1 and 0.4, respectively. As expected, we can have orbits which are similar to the segment and circle, shown in Figure~\ref{fig:SegmentEllipse} and Figure~\ref{fig:CircleEllipse}, respectively. 

\begin{figure}[H]
    \centering
    \includegraphics[width=0.9\linewidth]{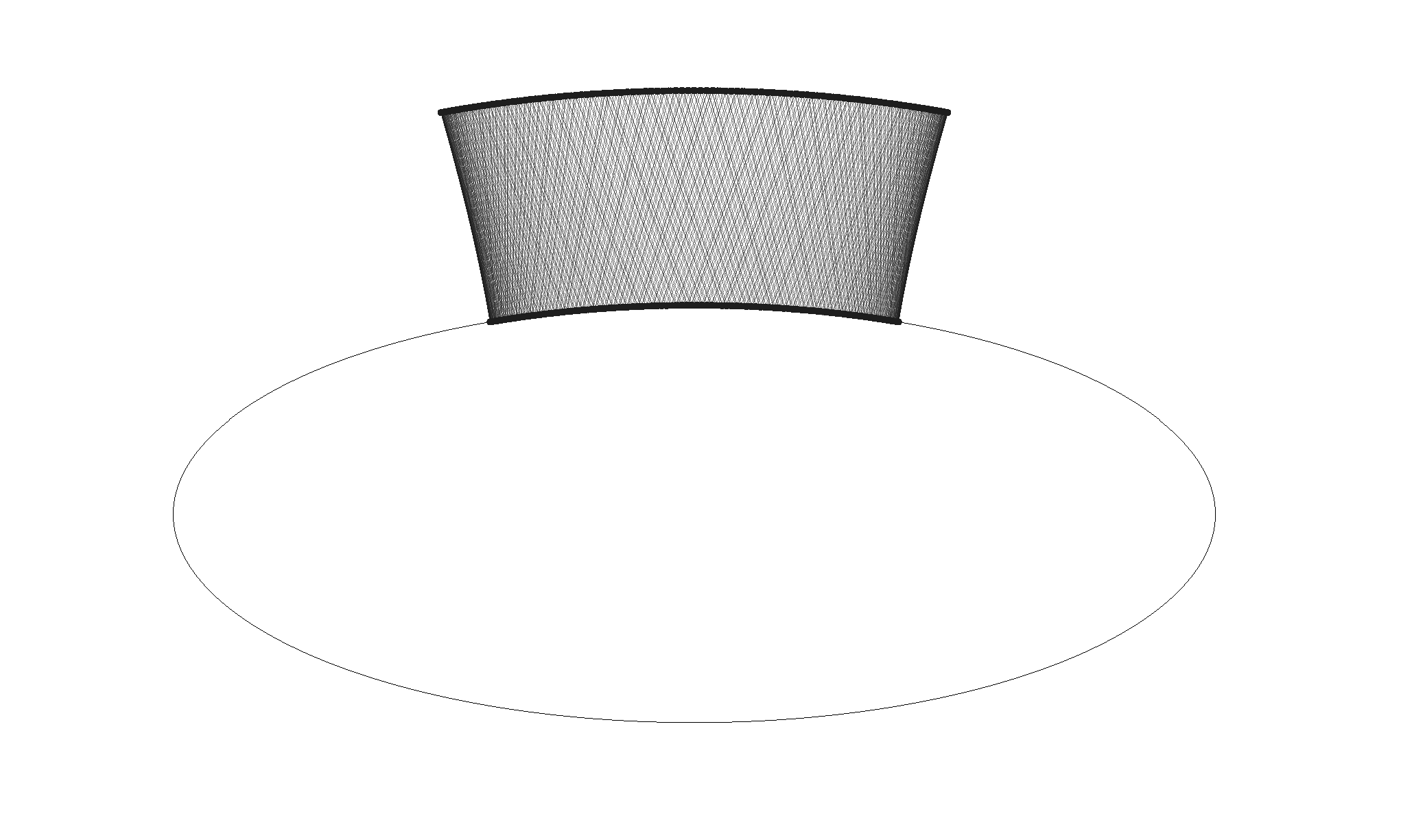}
    \caption{Segment-like  behavior}
    \label{fig:SegmentEllipse}
\end{figure}

\begin{figure}[H]
    \centering
    \includegraphics[width=0.9\linewidth]{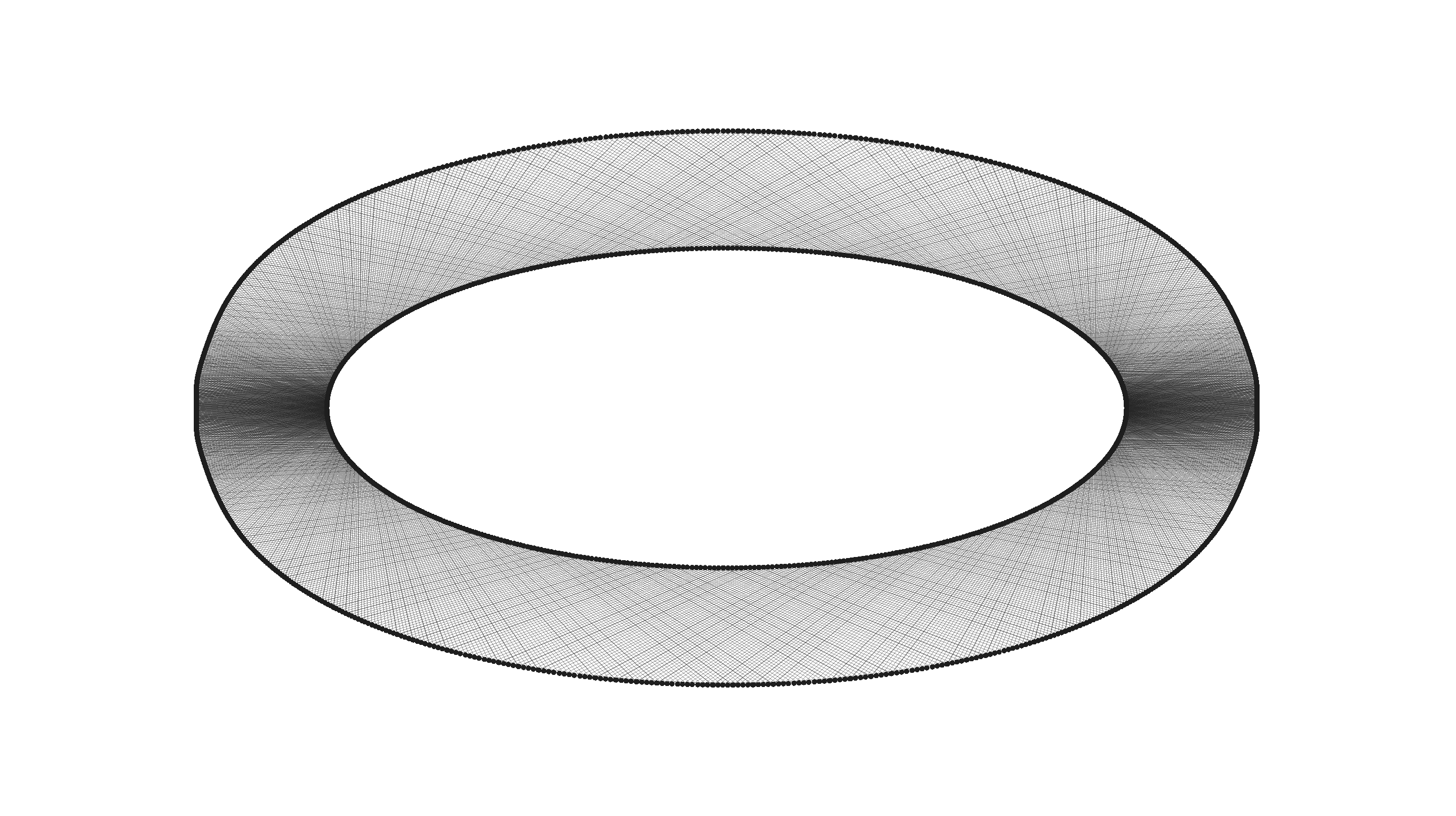}
    \caption{Circle-like behavior}
    \label{fig:CircleEllipse}
\end{figure}

We also have cases where the orbit closure is a periodic circle such as those on Figure~\ref{fig:FourPieceEllipse} and Figure~\ref{fig:ManyEllipse}.

\begin{figure}[H]
    \centering
    \includegraphics[width=0.9\linewidth]{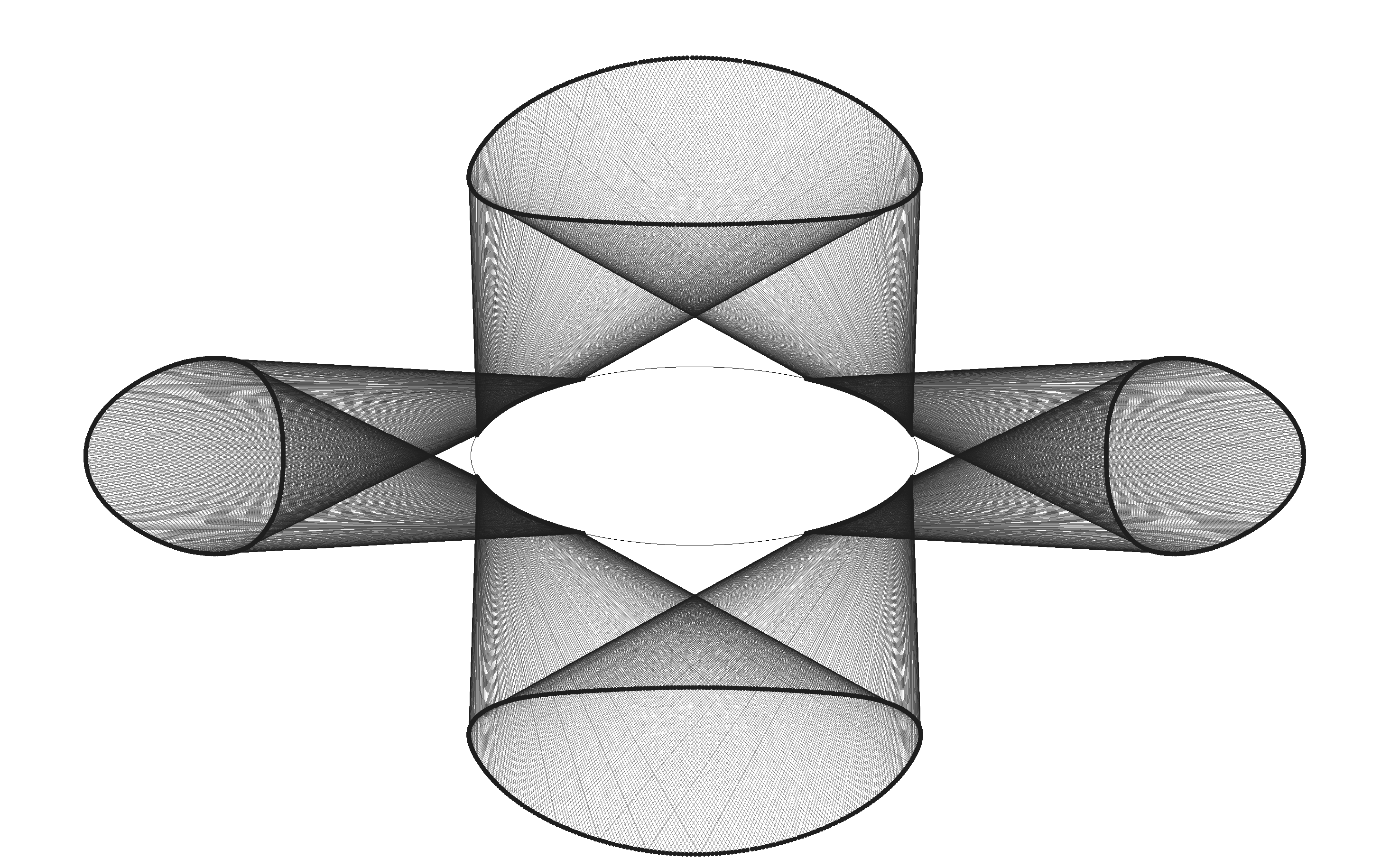}
    \caption{Four closed invariant curves}
    \label{fig:FourPieceEllipse}
\end{figure}

\begin{figure}[H]
    \centering
    \includegraphics[width=0.9\linewidth]{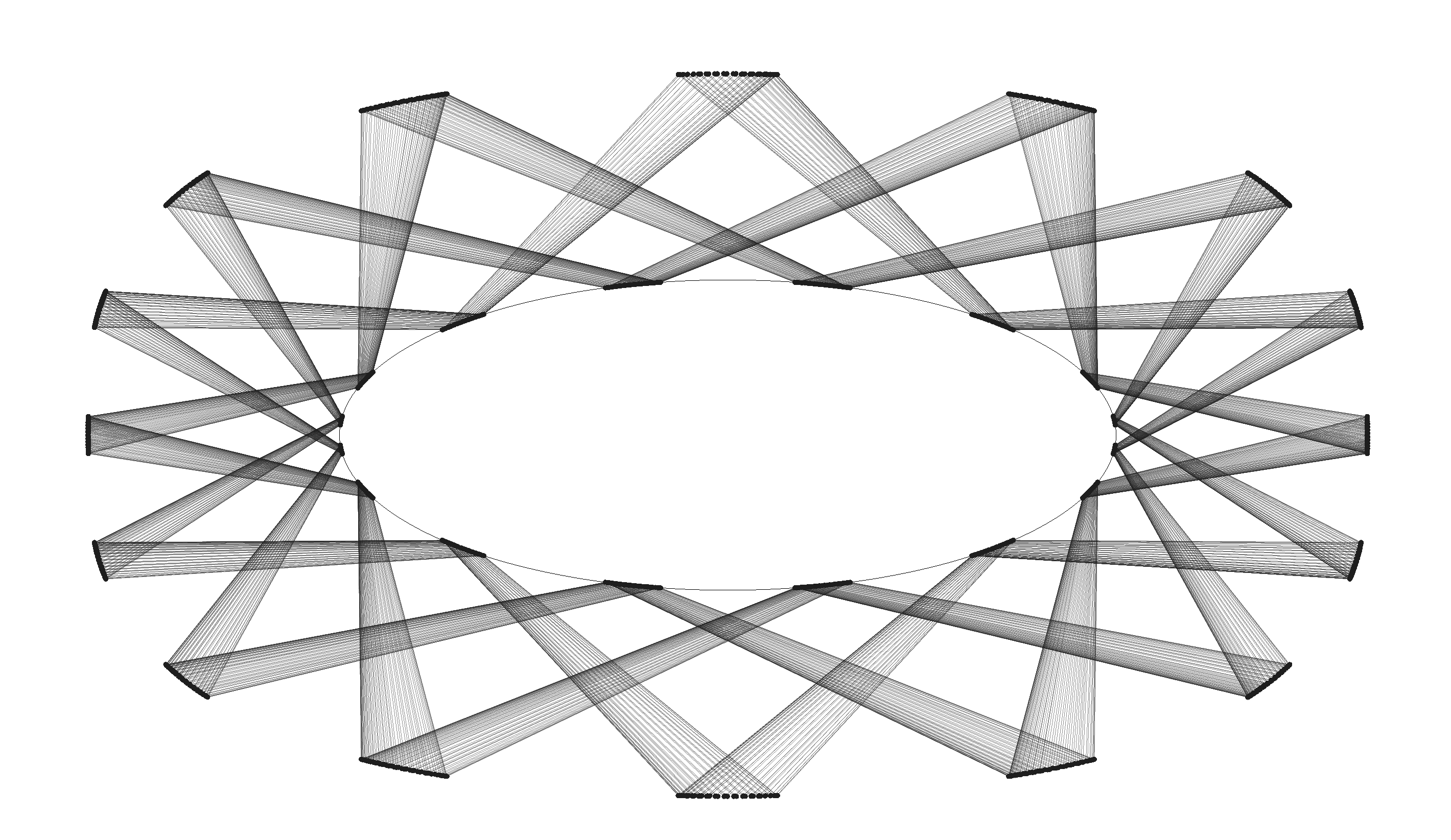}
    \caption{Many closed invariant curves}
    \label{fig:ManyEllipse}
\end{figure}

Finally, ``in between" the circle-like behavior and four closed curve behavior, we detect an orbit which appears to  fill up positive area domain, as shown in Figure~\ref{fig:PositiveAreaEllipse}. 

We notice that the types of orbits we observe for the parabola arc and the ellipse are the same.

\begin{figure}[H]
    \centering
    \includegraphics[width=0.9\linewidth]{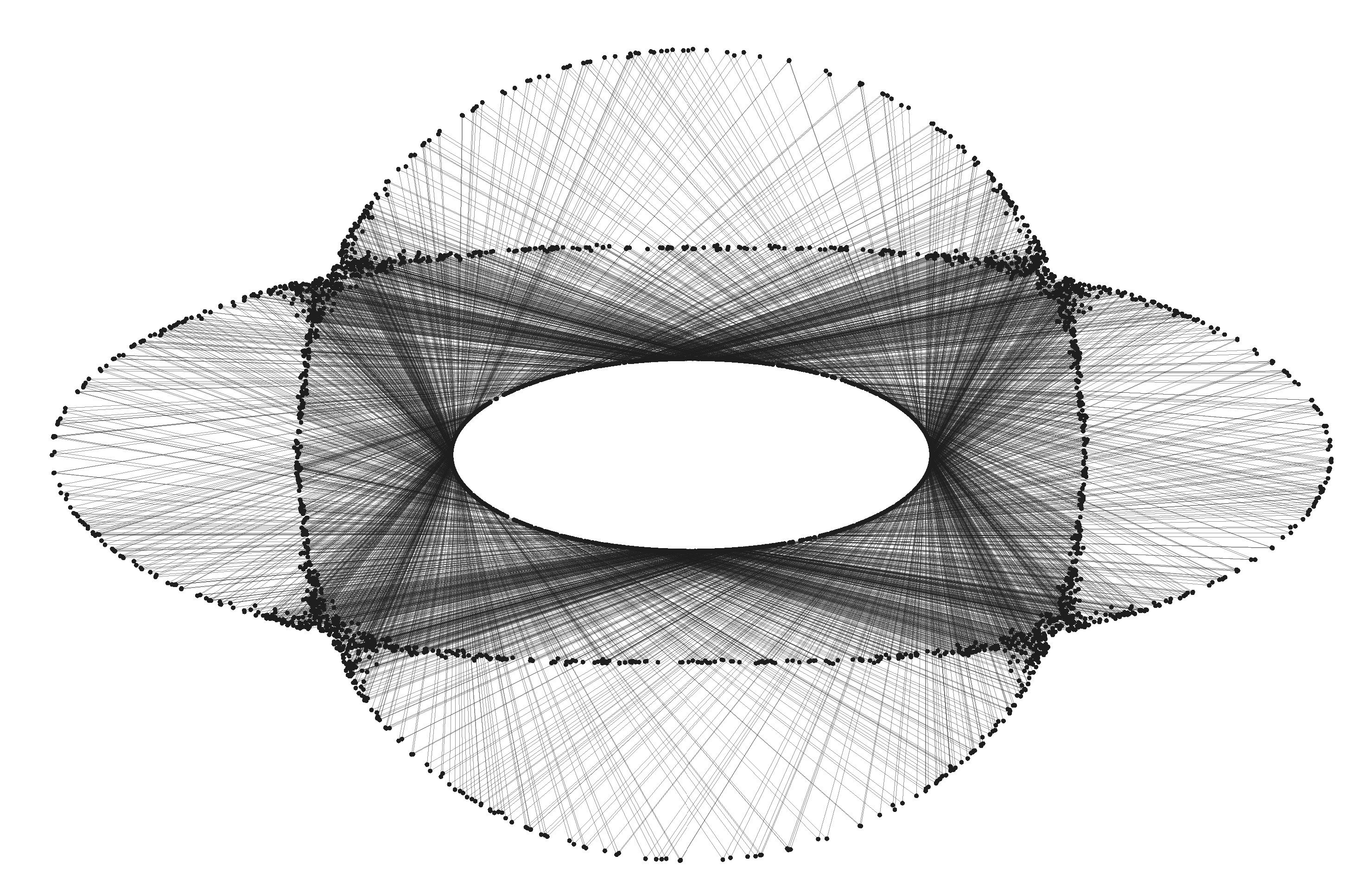}
    \caption{A chaotic orbit}
    \label{fig:PositiveAreaEllipse}
\end{figure}

\section*{Appendix A: The Conservative Property}

Here we verify that bouncing outer billiard dynamics $F$ preserves the Lebesgue measure. Let $dA$ be the standard 2-dimensional Lebesgue measure restricted to $\mathbb R^2$ and let $d\theta$ be the Lebesgue on the circle.

\begin{proposition}
    Assume that the boundary of $S$ is a $C^2$ curve then the restriction of $dA\otimes d\theta$ to $V_S$ is an infinite measure which is invariant under $F$.
\end{proposition}

\begin{remark}
 It is easy to verify invariance of the Lebesgue measure for polygons and seems likely to be true for any convex $S$, but we haven't verified it in such generality.
\end{remark}

\begin{proof}
    We begin by noticing that the proposition holds true if $S$ is a closed disc. Indeed, in this case it is easy to see that the visibility domain can be decomposed into circles on each of which $F$ is a rigid rotation preserving the length (conditional measure). Hence $F$ preserves $dA\otimes d\theta$.

    Now given a general domain $S$ with $C^2$ boundary we will verify that $F_S$ is measure preserving by checking that the Jacobian $JF=\det(DF)$ equals to 1. Let $(p,v)\in~V_S$, we can assume that $(p,v)$ is in fact in the interior of $V_S$ since $\partial V_S$ has measure zero. Therefore the ray starting at $v$ intersect $\partial S$ at the bounce point $w$ transversely. This implies that infinitesimal variations of $(p,v)$ result in infintesimal variations of $w$ of the same order of magnitude. 

    Consider the (unique) closed disc $D$ such that $\partial D$ is tangent to $\partial S$ at $w$ to the second order. Clearly, we have $F_S(p,v)=F_D(p,v)$. In fact, second order tangency ensures that $DF_S(p,v)=DF_D(p,v)$. Indeed, to see that first note that infinitesimal variation of $(p,v)$ results in infenitesimal variations of $w$ (on $\partial S$ and $\partial D$) agree up to second order. The angles of reflection of $w$ are controlled by the derivatives of $\partial S$ and $\partial D$ at $w$ and, hence, agree up to the first order (again due to second order tangency at $w$) and the claim follows. Hence
    we have 
    $$
    J F_S(p,v)=\det DF_S(p,v)=\det DF_D(p,v)=1,
    $$
    where the last equality is by measure-preserving property of $F_D$ pointed at out the beginning of the proof.

    It is easy to show that the total Lebesgue measure of $V_S$ is infinite when integrating in the correct order. For any angle $\theta$, there is an infinitely long strip of constant width $w$ of points that will hit the $S$, where $w$ is the length of $S$ projected to the axis perpendicular to $\theta$.
\end{proof}

Alternatively one can verify the above proposition by a direct calculation of $DF(p,v)$ without any reference to the disc case. Specifically we can center the $(x,y)$-coordinate system at the bounce point $w$ so that the $x$-axis is tangent to $\partial S$ and denote by $\theta$ the circular coordinate. We write $(p,v)=(a,b, \theta_0)$. Clearly $dA\otimes d\theta=dx\otimes dy\otimes d\theta$ and we need to verify that the Jacobian is 1 in $(x,y,\theta)$-coordinates. Prior to the last visibility angle reflection step we have  $(a,b,\theta_0)\mapsto (-a,b,-\theta_0)$ and a routine calculation gives the following expression for its derivative:
$$
\begin{pmatrix}
1+2kb & -2ka-\frac{2a}{b} & \frac{2c^2}{b}+2kc^2 \\
2ka & 1-\frac{2ka^2}{b} & \frac{2kac^2}{b}  \\
-2k & \frac{2ka}{b} & -1 - \frac{2kc^2}{b}
\end{pmatrix}
$$
where $k$ is the curvature at $w$ and $c=\sqrt{a^2+b^2}$. The determinant of this matrix is $-1$, which becomes 1, after composing with the reflection $\theta\mapsto const-\theta$ according to the visibility angle reflection rule.

\section*{Appendix B: Proof of Theorem~\ref{thmP}}\label{appendix}

First, we will find an explicit formula for $r^{-1}(w, d)$ for $w$ and $d$ lying on the ellipse. We have $w = a  \cos(\theta)$ and $d = b  \sin(\theta)$, meaning $\tan(\theta) = \frac{ad}{bw}$. This yields:
\begin{equation}\label{rInv}
    \theta = r^{-1}(w, d) = \arctan\left(\frac{ad}{bw}\right) + \pi  n(w)
\end{equation}
In this formula, we use

\begin{equation*}
    n(x) := \begin{cases}
        1 & x < 0 \\
        0 & x \ge 0
    \end{cases}
\end{equation*}
to compensate for the fact that $\arctan$ only outputs between $\frac{-\pi}{2}$ to $\frac{\pi}{2}$. This explicit formula has the slight flaw that it fails for $w = 0$, however it can be shown separately that this case matches the behavior of all other cases. Applying $r^{-1}$ to the right of both sides of the equation $\overline{f} = r^{-1} \circ f \circ r$ yields $\overline{f} \circ r^{-1}(w, d) = r^{-1} \circ f(w, d) = r^{-1}(w', d')$. Applying (\ref{rInv}) yields:

\begin{equation*}
    \overline{f}(\arctan\left(\frac{ad}{bw}\right) + \pi  n(w)) = \arctan\left(\frac{ad'}{bw'}\right) + \pi  n(w')
\end{equation*}
Thus $\overline{f}(\theta) = \theta + \varphi(w, d)$, where 

\begin{equation*}
    \varphi(w, d) = \arctan\left(\frac{ad'}{bw'}\right) + \pi  n(w') - \arctan\left(\frac{ad}{bw}\right) - \pi  n(w).
\end{equation*}
We will first show that $\varphi(w, d)$ is constant mod $\pi$. The terms $\pi  n(w')$ and $-\pi  n(w)$ are equivalent to zero mod $\pi$, so we will revisit these later.

Thus, we currently seek to show that $\arctan(\frac{ad'}{bw'}) - \arctan(\frac{ad}{bw})$ is constant mod $\pi$. We will use the arc-tangent subtraction formula $\arctan(x)-\arctan(y) = \arctan(\frac{x-y}{1+xy}) + m\pi$, where $m$ is either 0 or 1 depending on $x$ and $y$. The term $m\pi$ is equivalent to zero mod $\pi$ in all cases, so we will revisit this term later as well.

For the purposes of the following calculation, we will set $x = \frac{ad'}{bw'}$ and $y = \frac{ad}{bw}$. Our goal is to show that $\frac{x-y}{1+xy}$ is constant for any $w$ and $d$ on the fixed ellipse. First, since $x$ and $y$ each have a factor of $\frac{a}{b}$, which is constant for points on the ellipse, we can pull this out of the fraction: 

\begin{equation*}
    \frac{x-y}{1+xy} = \left(\frac{a}{b}\right)\left(\cfrac{\frac{d'}{w'} - \frac{d}{w}}{1+\frac{a^2dd'}{b^2ww'}}\right)
\end{equation*}

Next, multiply the numerator an denominator by $b^2ww'$ to get:

\begin{equation*}
    \left(\frac{a}{b}\right)\left(\frac{b^2d'w - b^2dw'}{b^2ww'+a^2dd'}\right)
\end{equation*}

We can pull out another $b^2$ to bring the total constant factored out to $ab$:

\begin{equation*}
    (ab)\left(\frac{d'w-dw'}{b^2ww'+a^2dd'}\right)
\end{equation*}

After replacing $w'$ and $d'$ with their equivalent expressions in terms of $w$, $h$, and $d$, then multiplying the numerator and denominator by $(w^2-d^2-h^2-1)$, we get: 

\begin{equation*}
    \frac{(ab)(2h^2w^2+2d^2)}{b^2(w^4+d^2w^2+h^2w^2+2dw^3-w^2-2dw) + a^2(d^4+2dh^2w+2d^3w+d^2w^2+d^2h^2-d^2)}
\end{equation*}

After substituting in the expressions for $a^2$ and $b^2$, multiplying the numerator and denominator by $(1-w^2)(h^2+d^2)$, and simplifying, we get:

\begin{equation*}
    \frac{2ab(1-w^2)(h^2+d^2)}{d^2h^2w^2+h^4w^2+h^2w^4+d^4+d^2h^2+d^2w^2-h^2w^2-d^2}
\end{equation*}

\vspace{0.2 cm}

Factoring the denominator yields: 

\begin{equation*}
    \frac{2ab(1-w^2)(h^2+d^2)}{(h^2w^2+d^2)(h^2+d^2)-(h^2w^2+d^2)(1-w^2)}
\end{equation*}

Finally, dividing the numerator and denominator by $(1-w^2)(h^2+d^2)$ gives:

\begin{equation*}
    \cfrac{2ab}{\frac{h^2w^2+d^2}{1-w^2} - \frac{h^2w^2+d^2}{(h^2+d^2)}} = \frac{2ab}{b^2-a^2}
\end{equation*}

Thus we end up with the equation:

\begin{equation*}
    \varphi = \arctan\left(\frac{2ab}{b^2-a^2}\right) mod \: \pi
\end{equation*}

Next, we will go back and carefully consider each of the extra terms we set aside earlier to show that $\varphi$ is actually constant mod $2\pi$. Each of these components individually may depend on $w, d, a, $ and $b$, but we will show that together they only depend on $a$ and $b$, which remain constant within an orbit.

We will begin with the term denoted as $m\pi$ earlier. Recall that this arose out of the extra term from the arc-tangent sum formula. Again using the definitions $x = \frac{ad'}{bw'}$ and $y = \frac{ad}{bw}$, we get that $m = 0$ if $-xy < 1$ and $m = 1$ if $-xy > 1$. This is equivalent to saying $m = 0$ if $1+xy > 0$ and $m = 1$ if $1+xy < 0$. After substituting in expressions to get $1+xy$ in terms of $w$, $h$, and $d$ as well as simplifying and factoring, we get: 

\begin{equation*}
    1+xy = \frac{(d^2+h^2w^2)(d^2+h^2+w^2-1)}{(d^2+h^2)(d^2w^2 + h^2w^2+2dw^3+w^4-2dw-w^2)}
\end{equation*}

Since we are only concerned about the sign of $1+xy$, and $\frac{d^2+h^2w^2}{d^2+h^2} \ge 0$, we can factor this out and ignore it. This leaves us with:

\begin{equation*}
    \frac{d^2+h^2+w^2-1}{d^2w^2 + h^2w^2+2dw^3+w^4-2dw-w^2} = \frac{d^2+h^2+w^2-1}{w^2(d^2+h^2+w^2-1) + 2dw^3-2dw}
\end{equation*}

This has the same sign as its reciprocal, which after simplification becomes:

\begin{equation*}
    w^2 + \frac{2dw^3-2dw}{d^2+h^2+w^2-1}
\end{equation*}

Next, it will benefit us to rewrite $d$ in terms of $w$, $a$, and $b$. Starting from the ellipse equation $\frac{w^2}{a^2} + \frac{d^2}{b^2} = 1$, we can derive the equation $d^2 = b^2 - \frac{w^2b^2}{a^2}$. Rewriting our previous expression yields:

\begin{equation*}
    w^2 + \cfrac{(2w^3-w)\left(\pm \sqrt{b^2 - \frac{w^2b^2}{a^2}}\right)}{b^2 - \frac{w^2b^2}{a^2} + h^2 + w^2 - 1}
\end{equation*}

Next, we will remove $h$ from the expression. Recall the relationship between $a^2$ and $b^2$ given by $b^2 = \frac{h^2a^2}{1-a^2}$. From this, we can derive $h^2 = \frac{b^2 - a^2b^2}{a^2}$. From here, we can perform a series of simplifications:

\begin{equation*}
\begin{aligned}
    w^2 + \cfrac{(2w^3-w)\left(\pm \sqrt{b^2 - \frac{w^2b^2}{a^2}}\right)}{b^2 - \frac{w^2b^2}{a^2} + h^2 + w^2 - 1} &= w^2 + \cfrac{2w(w^2-1)\left(\pm \sqrt{b^2 - \frac{w^2b^2}{a^2}}\right)}{\frac{b^2}{a^2} - \frac{w^2b^2}{a^2} + \frac{w^2a^2}{a^2} - \frac{a^2}{a^2}} \\
    &= w^2 + \cfrac{2wa^2(w^2-1)\left(\pm \sqrt{\frac{b^2(a^2 - w^2)}{a^2}}\right)}{(w^2-1)(a^2-b^2)} \\
    &= w^2 + \cfrac{2wab(\pm \sqrt{a^2 - w^2})}{a^2 - b^2} \\
    &= \cfrac{a^2w^2 - b^2w^2 + 2wab(\pm \sqrt{a^2 - w^2})}{a^2-b^2}
\end{aligned}
\end{equation*}

Next, we want to find the zeros of this expression with respect to $w$. Clearly, $w = 0$ is a zero. To find other zeros, we will set the numerator of our expression equal to zero and assume $w \ne 0$. $a^2w^2 -b^2w^2+ 2wab(\pm \sqrt{a^2 - w^2}) = 0 \implies a^2w -b^2w+ 2ab(\pm \sqrt{a^2 - w^2}) = 0$. After rearranging and squaring both sides, we get:

\begin{equation*}
    a^2 - w^2 = \frac{b^4w^2-2a^2b^2w^2+a^4w^2}{4a^2b^2}
\end{equation*}

Solving for $w$ yields:

\begin{equation*}
w = \pm \frac{2a^2b}{a^2+b^2}
\end{equation*}

From this point forward, we will assume $d$ is positive. The calculations play out similarly if $d$ is negative, and the results will be given with the positive case. This means our fraction is zero when $a^2w -b^2w+ 2ab\sqrt{a^2 - w^2} = 0$, removing the plus or minus present earlier. We have the possible zeros of $\pm \frac{2a^2b}{a^2+b^2}$, but determining which is a true zero will depend on the values of $a$ and $b$. This is because $\sqrt{a^2 - w^2}$ becomes $\frac{a(a^2-b^2)}{a^2+b^2}$ if $a > b$ or $\frac{a(b^2-a^2)}{a^2+b^2}$ if $a < b$. This means that if $a > b$ we have $\frac{-2a^2b}{a^2+b^2}$ is a zero, whereas if $b > a$ we have $\frac{2a^2b}{a^2+b^2}$ is a zero. 

To find the sign of the expression, we can solve the derivatives at the zeros:

\begin{equation}\label{derivExp}
\begin{aligned}
    \frac{d}{dw}\left(\cfrac{a^2w^2 - b^2w^2 + 2wab\sqrt{a^2 - w^2}}{a^2-b^2}\right) = \cfrac{2a^2w-2b^2w+2ab\sqrt{a^2-w^2} - \frac{2abw^2}{a^2-w^2}}{a^2-b^2}
\end{aligned}
\end{equation}

We will first analyze the derivative values for $a > b$. The denominator is positive in this case, and we have that the derivative is positive at $w=0$. For $w = \frac{-2a^2b}{a^2+b^2}$, we have that the numerator of (\ref{derivExp}) becomes:

\begin{equation*}
\frac{2a^2b(b^2-a^2)}{a^2+b^2} + \frac{-2abw^2}{a^2-w}
\end{equation*}

The second term of this expression is always negative. Since we assumed $a > b$, the first term is also negative, meaning the derivative is negative for this $w$-value.

In the case where $a < b$, the denominator is always negative, and we have that the derivative is negative at $w=0$. For $w = \frac{2a^2b}{a^2+b^2}$, the numerator becomes:

\begin{equation*}
\frac{2a^2b(a^2-b^2)}{a^2+b^2} + \frac{-2abw^2}{a^2-w}
\end{equation*}

This is again negative, but the derivative is positive since the denominator of (\ref{derivExp}) is negative. In summary, we have the following results:
\vspace{0.1in}

\noindent\fbox{\begin{minipage}{\textwidth}
\paragraph{In $d > 0$ Case:} 
$\;$ If $a < b$ and $w \in (\frac{2a^2b}{a^2+b^2}, 1] \cup [-1, 0)$, then $1+xy > 0$. 

$\;$ If $a < b$ and $w \in (0, \frac{2a^2b}{a^2+b^2})$, then $1+xy < 0$.

$\;$ If $a > b$ and $w \in [-1, \frac{-2a^2b}{a^2+b^2}) \cup (0, 1]$, then $1+xy > 0$.

$\;$ If $a > b$ and $w \in (\frac{-2a^2b}{a^2+b^2}, 0)$, then $1+xy < 0$.
\end{minipage}}
\\~\\
The case when $d < 0$ works similarly, with results as follows:
\\~\\
\noindent\fbox{\begin{minipage}{\textwidth}
\paragraph{In $d < 0$ Case:} 
$\;$ If $a < b$ and $w \in [-1, \frac{-2a^2b}{a^2+b^2}) \cup (0, 1]$, then $1+xy > 0$.

$\;$ If $a < b$ and $w \in (\frac{-2a^2b}{a^2+b^2}, 0)$, then $1+xy < 0$.

$\;$ If $a > b$ and $w \in (\frac{2a^2b}{a^2+b^2}, 1] \cup [-1, 0)$, then $1+xy > 0$.

$\;$ If $a > b$ and $w \in (0, \frac{2a^2b}{a^2+b^2})$, then $1+xy < 0$. 
\end{minipage}}
\\~\\

Next, we will tackle the $\pi  n(w')$ term. Recall that $n(w')$ is defined to be 1 when $w'$ is negative and 0 otherwise. Thus, our next goal is to determine the sign of $w'$ under all possible conditions. We will again take $d > 0$ and present the results for the $d < 0$ case later.

\begin{equation}
\begin{aligned}
    w' &= \frac{w^3+d^2w+h^2w+2dw^2-w-2d}{w^2-d^2-h^2-1} \\
    &= \cfrac{w^3-\frac{b^2w^3}{a^2}+\frac{b^2w}{a^2}+2w^2\sqrt{b^2-\frac{b^2w^2}{a^2}}-w-2\sqrt{b^2-\frac{b^2w^2}{a^2}}}{w^2+\frac{b^2w^2}{a^2}-\frac{b^2}{a^2}-1} \\
    &= \frac{(w^2-1)(a^2w-b^2w+2ab\sqrt{a^2-w^2})}{(w^2-1)(a^2+b^2)} \\
    &= \frac{a^2w-b^2w+2ab\sqrt{a^2-w^2}}{a^2+b^2}
\end{aligned}
\end{equation}

We have already examined the zeros of this expression when working through the $m\pi$ term. It has a zero at $w = \frac{2a^2b}{a^2+b^2}$ when $a < b$ and one at $\frac{-2a^2b}{a^2+b^2}$ when $a > b$. Notably, this expression does not have a zero at $w = 0$ like the previous. In this case we have that the derivative is negative at $w = \frac{2a^2b}{a^2+b^2}$ when $a < b$ and positive at $\frac{-2a^2b}{a^2+b^2}$ when $a > b$.
\\~\\
\noindent\fbox{\begin{minipage}{\textwidth}
\paragraph{In $d > 0$ Case:} 
$\;$ If $a < b$ and $w \in [-1, \frac{2a^2b}{a^2+b^2})$, then $w' > 0$.

$\;$ If $a < b$ and $w \in (\frac{2a^2b}{a^2+b^2}, 1]$, then $w' < 0$.

$\;$ If $a > b$ and $w \in (\frac{-2a^2b}{a^2+b^2}, 1]$, then $w' > 0$.

$\;$ If $a > b$ and $w \in [-1, \frac{-2a^2b}{a^2+b^2})$, then $w' < 0$.
\end{minipage}}
\\~\\
Similarly, if $d < 0$, we get:
\\~\\
\noindent\fbox{\begin{minipage}{\textwidth}
\paragraph{In $d < 0$ Case:} 
$\;$ If $a < b$ and $w \in [-1, \frac{-2a^2b}{a^2+b^2})$, then $w' > 0$.

$\;$ If $a < b$ and $w \in (\frac{-2a^2b}{a^2+b^2}, 1]$, then $w' < 0$.

$\;$ If $a > b$ and $w \in (\frac{2a^2b}{a^2+b^2}, 1]$, then $w' > 0$.

$\;$ If $a > b$ and $w \in [-1, \frac{2a^2b}{a^2+b^2})$, then $w' < 0$.
\end{minipage}}
\\~\\
Finally, we have the $-\pi n(w)$ term, which requires no extra analysis since our results are currently allowed to be dependent on $w$.

The final step in this proof is to check how many $\pi$ terms are being added for each initial condition for $w$ and $d$, as well as $a$ and $b$ values. This will be omitted since it solely involves going through each relevant interval for $w$ and $d$ for both the $a > b$ and $a < b$ case. The results are as follows:

\begin{equation*}
    \varphi = \begin{cases}
        \arctan(\frac{2ab}{b^2-a^2}) + \pi & a < b \\
        \arctan(\frac{2ab}{b^2-a^2}) & a > b \\
        \frac{-\pi}{2} & a=b        
    \end{cases}
\end{equation*}

Taking the derivative of $\varphi$ with respect to $a$ yields:

\begin{equation*}
    \varphi'(a) = \cfrac{2b}{b^2+a^2}
\end{equation*}
This is clearly positive for all values $a > 0$ since $b$ is nonzero and positive for such $a$. \qed

\bibliographystyle{plain}
\nocite{*}
\bibliography{sn-bibliography.bib}% common bib file

\begin{thebibliography}{1}

\bibitem{bib1}
J.~Moser.
\newblock {\em Stable and Random Motions in Dynamical Systems}, volume~77 of {\em Annals of Mathematics Studies}.
\newblock Princeton University Press, 1973.

\bibitem{bib2}
J.~Moser.
\newblock Is the solar system stable?
\newblock {\em Mathematical Intelligencer}, \textbf{1}:65--71, 1978.

\bibitem{bib3}
B.~H. Neumann.
\newblock Sharing ham and eggs.
\newblock {\em Iota.}, pages 14--18, 1958.

\bibitem{Tab}
S.~Tabachnikov.
\newblock On the dual billiard problem.
\newblock {\em Adv. Math.}, 115(2):221--249, 1995.

\bibitem{bib4}
S.~Tabachnikov.
\newblock The (un)equal tangents problem.
\newblock {\em The American Mathematical Monthly}, 119(5):398--405, 2012.

\end{thebibliography}

% RED CURVE IN PARABOLA HAS CONDITIONS HEIGHT = 1, X = -1, Y = 1.5, M = -0.46
% BLUE CURVE IN PARABOLA HAS CONDITIONS HEIGHT = 1, X = -1, Y = 2, M = -2
% GREEN CURVE IN PARABOLA HAS CONDITIONS HEIGHT = 1, X = -0.77, Y = 1.17, M = -0.5
% purple CURVE IN PARABOLA HAS CONDITIONS HEIGHT = 1, X = 0.286, Y = 1.3, M = -1.83
% DARK purple CURVE IN PARABOLA HAS CONDITIONS HEIGHT = 1, X = -0.0001, Y = 2, M = -2

% RED CURVE IN SECOND PARABOLA HAS CONDITIONS HEIGHT = 0.5, X = -1.17, Y = 1.86, M = -0.9
% BLUE CURVE IN SECOND PARABOLA HAS CONDITIONS HEIGHT = 0.5, X = -1.151, Y = 1.86, M = -0.9
% DARK purple CURVE IN SECOND PARABOLA HAS CONDITIONS HEIGHT = 0.5, X = -2.07, Y = 2.5, M = -0.85
% GREEN CURVE IN SECOND PARABOLA HAS CONDITIONS HEIGHT = 0.5, X = -0.8, Y = 1.2, M = -0.85

% DARK PURPLE CURVE IN THIRD PARABOLA HAS CONDITIONS HEIGHT = 0.3, X = -3.34, Y = 2.17, M = -0.499
% BLUE CURVE IN THIRD PARABOLA HAS CONDITIONS HEIGHT = 0.3, X = -2.94, Y = 2.13, M = -1.0
% GREEN CURVE IN THIRD PARABOLA HAS CONDITIONS HEIGHT = 0.3, X = -1.28, Y = 1.01, M = -0.6
% RED CURVE IN THIRD PARABOLA HAS CONDITIONS HEIGHT = 0.3, X = -2.45, Y = 0.9, M = -0.6

% GREEN CURVE IN SQUARE HAS CONDITIONS X = -1.0, Y = 2.0, M = -2.0
% BLUE CURVE IN SQUARE HAS CONDITIONS X = 1.8, Y = 2.276, M = 0.7
% RED CURVE IN SQUARE HAS CONDITIONS X = -1.081, Y = 2.0, M = -0.5
% MAGENTA CURVE IN SQUARE HAS CONDITIONS X = 1.4, Y = 2.025, M = 0.5
% DARK PURPLE CURVE IN SQUARE HAS CONDITIONS X = -3.8, Y = 3.2, M = -1.0

\end{document}